\documentclass{amsart}

\usepackage{amssymb}
\usepackage{amsmath}
\usepackage{amsthm}
\usepackage{amsbsy}
\usepackage{bm}
\usepackage{hyperref}
\date{\today}
\usepackage{cite}
\usepackage{array}
\usepackage{xcolor}
\usepackage{tikz}
\usepackage{graphics}
\usepackage{subcaption}
\theoremstyle{theorem}
    \newtheorem{theorem}{Theorem}
    \newtheorem{lemma}[theorem]{Lemma}

\theoremstyle{definition} % For roman text in the body
    \newtheorem{definition}[theorem]{Definition}
    
    \newtheorem{result}[theorem]{Result}
    \newtheorem{remark}[theorem]{Remark}
    \newtheorem{example}[theorem]{Example}
    \newtheorem{exercise}[theorem]{Exercise}

\def\suchthat{\; : \;}

\def\tends{\rightarrow}

\def\l{\left}
\def\r{\right}
\def\<{\langle}
\def\>{\rangle}

\newcommand{\E}{\mbox{E}}

\newcommand\Tr{{\mbox{Tr}}}

\newcommand\mnote[1]{} %off
\newcommand\be{\begin{equation*}}

\newcommand\ee{\end{equation*}}

\newcommand\ben{\begin{equation}}
\newcommand\een{\end{equation}}
\newcommand\bes{\begin{eqnarray*}}
\newcommand\ees{\end{eqnarray*}}

\newcommand\bex{\begin{exercise}}
\newcommand\eex{\end{exercise}}
\newcommand\beg{\begin{example}}
\newcommand\eeg{\end{example}}
\newcommand\benu{\begin{enumerate}}
\newcommand\eenu{\end{enumerate}}
\newcommand\beit{\begin{itemize}}
\newcommand\eeit{\end{itemize}}
\newcommand\berk{\begin{remark}}
\newcommand\eerk{\end{remark}}
\newcommand\bdefn{\begin{defintion}}
\newcommand\edefn{\end{definition}}
\newcommand\bthm{\begin{theorem}}
\newcommand\ethm{\end{theorem}}
\newcommand\bprf{\begin{proof}}
\newcommand\eprf{\end{proof}}
\newcommand\blem{\begin{lemma}}
\newcommand\elem{\end{lemma}}

\newcommand{\Cov}{\mbox{\rm Cov}}
\newcommand{\sm}{{\raise0.3ex\hbox{$\scriptstyle \setminus$}}}

\def\l{\left}
\def\r{\right}

\def\tends{\rightarrow}

\def\CHI{\mathchoice%
{\raise2pt\hbox{$\chi$}}%
{\raise2pt\hbox{$\chi$}}%
{\raise1.3pt\hbox{$\scriptstyle\chi$}}%
{\raise0.8pt\hbox{$\scriptscriptstyle\chi$}}}
\def\smalloplus{\raise1pt\hbox{$\,\scriptstyle \oplus\;$}}

\title[linear eigenvalue statistics of symmetric circulant matrices]{Fluctuation of eigenvalues of symmetric circulant matrices with  independent entries} 
\author{Shambhu Nath Maurya}
\address{Department of Mathematics\\
	Indian Institute of Technology Bombay\\
	Powai, Mumbai, Maharashtra 400076, India}

\email{snmaurya [at] math.iitb.ac.in}

\author{Koushik Saha}
\address{Department of Mathematics\\
        Indian Institute of Technology Bombay\\
         Powai, Mumbai, Maharashtra 400076, India}

\email{koushik.saha [at] iitb.ac.in}

\date{\today}
\thanks{The work of Shambhu Nath Maurya is partially supported by UGC Doctoral Fellowship, India and the work of Koushik Saha is partially supported by MATRICS grant of SERB, Department of Science and Technology, Government of India.}
\begin{document}

\begin{abstract}
In this article, we study the fluctuation of linear eigenvalue statistics of symmetric circulant matrices $(SC_n)$ with independent entries which satisfy some moment conditions.
% are independent random variables with some moment conditions. First we derive trace formula for $SC_n$ and  we study the fluctuation of the linear statistics, $\frac{1}{\sqrt{n}} \Tr \phi(SC_n)$ for some nice test function $\phi$. 
We show that $\frac{1}{\sqrt{n}} \Tr \phi(SC_n)$ obey the central limit theorem (CLT) type result, where 
$\phi$ is a nice test function.

%We show that $\frac{1}{\sqrt{n}} \Tr \phi(SC_n)$ obey the central limit theorem (CLT) type result where the limit law are Gaussian which are not independent.
%We discuss the weak convergence of linear eigenvalues statistics of symmetric circulant matrices for general entries which have some moment conditions. 
%%By using method of moments and some combinatorial technique, we show that the linear statistics of eigenvalues of reverse circulant matrices obey the central limit theorem (CLT) type result.
%We derive central limit theorem (CLT) type result for linear eigenvalue statistics of symmetric circulant matrices. To prove this, we express linear eigenvalue statistics in terms of trace of symmetric circulant matrix and finally we apply method of moments and some combinatorial technique.
%%We  show that the linear statistics of eigenvalues of random circulant matrices obey the Gaussian central limit theorem for a large class of input sequences.
\end{abstract}

\maketitle

\noindent{\bf Keywords :} Symmetric circulant matrix, linear statistics of eigenvalues, weak convergence, central limit theorem, Trace formula, Wick's formula.
\section{ introduction and main results}
Let $A_n$ be an $n\times n$ matrix with real or complex entries. The linear statistics of
eigenvalues $\lambda_1,\lambda_2,\ldots, \lambda_n$ of $A_n$ is a
function of the form
\begin{equation} \label{eqn:1}
\frac{1}{n}\sum_{k=1}^{n}f(\lambda_k)
\end{equation}
where $f$ is some  fixed function. The function $f$ is known as the test function. One of the interesting 
object  to study in random matrix theory  is the fluctuation of linear
statistics of eigenvalues of random matrices.  The study of fluctuation of linear statistics of eigenvalues was  initiated by Arharov \cite{arharov} in 1971 for sample covariance matrices. In 1975 Girko \cite{girko} studied the central limit theorem (CLT) of the traces of the Wigner and sample covariance matrices using martingale techniques. In
1982, Jonsson \cite{jonsson} proved the 
CLT of linear eigenvalue statistics for Wishart matrices using method of moments.  After that the fluctuations of
eigenvalues for various  random matrices have been extensively studied by various people.  
%The CLT for the linear statistics of eigenvalues of various matrix models has been studied by various people 
For new results on fluctuations  of linear eigenvalue statistics of  Wigner and sample covariance matrices,  see  \cite{johansson1998}, \cite{soshnikov1998tracecentral}, \cite{bai2004clt}, \cite{lytova2009central}, \cite{shcherbina2011central}.  For   band and sparse  random matrices, see  \cite{anderson2006clt},  \cite{jana2014}, \cite{li2013central},  \cite{shcherbina2015} and for Toeplitz and band Toeplitz matrices, see  \cite{chatterjee2009fluctuations} and \cite{liu2012}.

In a recent article \cite{adhikari_saha2017},  the CLT for linear eigenvalue statistics has been established in total variation norm for   circulant, symmetric circulant and reverse circulant matrices   with Gaussian entries. In a subsequent article \cite{adhikari_saha2018}, the authors extended their results for independent entries which are smooth functions of Gaussian variables.  Here we consider the fluctuation problem for  symmetric circulant matrices with general entries which are independent and  satisfy some moment condition.
 
%In this paper, we are interested in fluctuation problem for reverse circulant matrix with general entries which which are independent and  satisfy some moment condition, will be given in Theorem \ref{thm:symcirpoly}. This article is inspired by [liu sun wand 2012] and \cite{adhikari_saha2018}. In [liu sun wand 2012], authors have done fluctuation problem for band Toeplitz matrix with general entries which which are independent and satisfy some moment condition. In \cite{adhikari_saha2018} authors have done fluctuation problem for random circulant matrix, where entries are also general and independent, and have same moment conditions, that we are going to consider.   

A sequence is said to be an {\it input sequence} if the matrices are constructed from the given sequence. We consider the input sequence of the form $\{x_i: i\geq 0\}$
%The circulant is constructed from this given input sequence, 
and the symmetric circulant matrix is defined as 
$$
SC_n=\left(\begin{array}{cccccc}
x_0 & x_1 & x_2 & \cdots & x_{2} & x_1 \\
x_1 & x_0 & x_1 & \cdots & x_{3} & x_{2}\\
x_{2} & x_1 & x_0 & \cdots & x_{4} & x_{3}\\
\vdots & \vdots & {\vdots} & \ddots & {\vdots} & \vdots \\
x_1 & x_2 & x_3 & \cdots & x_1 & x_0
\end{array}\right).
$$
For $j=1,2,\ldots, n-1$, its $(j+1)$-th row is obtained by giving its $j$-th row a right circular shift by one positions and the (i,\;j)-th element of the matrix is $x_{\frac{n}{2}-|\frac{n}{2}-|i-j||}$. Also note that the symmetric circulant matrix is a Toeplitz matrix with the restriction that $x_{n-j}=x_j$.
% we write $x_n$ instead of $x_0$.

Now we consider linear eigenvalue statistics as defined in \eqref{eqn:1} for $SC_n$ with test function $f(x)=x^{p}$, $p\geq 2$. Therefore
$$ \sum_{k=1}^{n}f(\lambda_k)= \sum_{k=1}^{n}(\lambda_k)^{p}= \Tr(SC_n)^{p},$$
where $\lambda_1,\lambda_2,\ldots,\lambda_n$ are the eigenvalues of $SC_n$. We scale and centre $\Tr(SC_n)^{p}$ to study its fluctuation, and define 
\begin{equation}\label{eqn:SCw_p}
w_p := \frac{1}{\sqrt{n}} \bigl\{ \Tr(SC_n)^{p} - \E[\Tr(SC_n)^{p}]\bigr\}. 
\end{equation}
For a given real polynomial  $$Q(x)=\sum_{k=1}^da_kx^{k}$$ 
with degree $d$ where $d\geq 2$, we define 
\begin{equation}\label{eqn:RCw_Q}
w_Q := \frac{1}{\sqrt{n}} \bigl\{ \Tr(Q(SC_n)) - \E[\Tr(Q(SC_n))]\bigr\}. 
\end{equation}
 
Note that $w_Q$ and $w_p$ depends on $n$. But we suppress $n$  to keep the notation simple. In our first result, we calculate the covariance between $w_p$ and $w_q$ as $n \tends \infty$. 
\begin{theorem}\label{thm:symcircovar}
 Suppose  $SC_n$ is the symmetric circulant matrix with independent input sequence $\{\frac{X_i}{\sqrt n}\}_{i\geq 0}$ such that
 \begin{equation}\label{eqn:condition}
\E(X_i)=0, \E(X_i^2)=1,  \E(X_i^4)=\E(X_1^4)\  \mbox{and}\  \sup_{i\geq 1}\E(|X_i|^k)=\alpha_k<\infty \ \mbox{for}\ k\geq 3.
\end{equation}  
 Then for $p,q \geq2$,
 \begin{align} \label{eqn:sigma_p,q}
	\sigma_{p,q}&:=\lim_{n\to\infty}  \Cov\big(w_p,w_q\big) \nonumber \\ 
	&=  \left\{\begin{array}{ll} 	 
		 	\displaystyle\frac{a_1}{2^{\frac{p+q-4}{2}}  } (\E X^4_1- 1) + \sum_{r=2}^{ \min\{ \frac{p}{2},\frac{q}{2} \} } \frac{a_r}{2^{\frac{p+q-4r}{2}}  } \sum_{s=0}^{2r}\binom{2r}{s}^2 s!(2r-s)! \ h_{2r}(s) & \text{if}\ p,  q \mbox{ both are even,}\\\\
		 	\displaystyle\sum_{r=0}^{ \min\{ \frac{p-1}{2},\frac{q-1}{2} \} }  \frac{b_r}{2^{\frac{p+q-4r-2}{2}}  } \sum_{s=0}^{2r+1}\binom{2r+1}{s}^2 s!(2r+1-s)! \ h_{2r+1}(s) \\		 	
		\displaystyle+ pq \binom{p-1}{(p-1)/2}\binom{q-1}{(q-1)/2} \l((p-1)/2\r)! \l((q-1)/2\r)!\frac{1}{2^{(\frac{p+q}{2}-1)}} 	 & \text{if}\ p,  q \mbox{ both are odd,}\\\\
			0 & \text{otherwise}, 	 	 
		 	  \end{array}\right.		
		 	  %\\
		 	  %&= \sigma_{p,q},\ \mbox{ say},
\end{align}
where $a_r$ and $b_r$ are appropriate constants, will be given in proof, and $h_d(s)$ is given as
$$h_d(s)=
 \frac{1}{(d-1)!}\sum_{i=-\lceil\frac{d-s}{2}\rceil}^{\lfloor\frac{s}{2}\rfloor}\sum_{j=0}^{2i+d-s}(-1)^q\binom{d}{j}\l(\frac{2i+d-s-j}{2}\r)^{d-1}.$$
 \end{theorem}
If $p=q$ then we denote $\sigma_{p,q}$ by $\sigma^{2}_{p}$.
In our second result, we see the fluctuation of linear  eigenvalue statistics of symmetric circulant matrices with polynomial test functions.
\begin{theorem}\label{thm:symcirpoly} 
Suppose  input entry of $SC_n$ is independent sequence $\{\frac{X_i}{\sqrt n}\}_{i\geq 1}$ which satisfy \ref{eqn:condition}.
%\begin{equation}\label{eqn:condition}
%\E(x_i)=0, \E(x_i^2)=1 \ \mbox{and}\  \sup_{i\geq 1}\E(|x_i|^k)=\alpha_k<\infty \ \mbox{for}\ k\geq 3.
%\end{equation}  
Then, as $n\to \infty$,
\begin{align*}
w_Q \stackrel{d}{\longrightarrow} N(0,\sigma_{Q}^2).
\end{align*}
In particular, for $Q(x)= x^{p}$
\begin{align*}
 w_p \stackrel{d}{\longrightarrow} N(0,\sigma_{p}^2),
\end{align*}
where 
$$\sigma_{Q}^2= \sum_{\ell=1}^d \sum_{k=1}^d a_{\ell} a_k \sigma_{\ell,k}, \ \  \sigma^2_{p}= \sigma_{p,p}$$
and  $\sigma_{p,q}$ is as given in (\ref{eqn:sigma_p,q}). 
%Here $\sigma_{p}$ denotes the standard deviation.
\end{theorem}

%\begin{remark}
%	In the above theorems we have considered the fluctuation of $w_p$ for  $p\geq1$. For $p=0$, 
%	\begin{align*}
%	w_0 = \frac{1}{\sqrt{n}} \bigl\{ \Tr(I) - \E[\Tr(I)]\bigr\} = \frac{1}{\sqrt{n}}[n-n] = 0
%	\end{align*}
%	and hence it has no fluctuation.
%	So we ignore this case, that is, for $p=0$
	
\begin{remark}
In the above theorems we have considered the fluctuation of $w_p$ for  $p\geq2$. For $p=0$, 
\begin{align*}
w_0 = \frac{1}{\sqrt{n}} \bigl\{ \Tr(I) - \E[\Tr(I)]\bigr\} = \frac{1}{\sqrt{n}}[n-n] = 0
\end{align*}
and hence it has no fluctuation.
For $p=1$,
\begin{align*}
w_1 = \frac{1}{\sqrt{n}} \bigl\{ \Tr(SC_n) - \E[\Tr(SC_n)]\bigr\} = \frac{1}{\sqrt{n}} \big[n \frac{X_0}{\sqrt n} - \E(n\frac{X_0}{\sqrt n})\big] = X_0,
\end{align*}
as $\E(X_0)=0$. So $w_1$ is distributed as $X_0$ and its distribution does not depend on $n$. So we ignore these two cases, for $p=0$ and $p=1$.  
\end{remark}	
	
%	For $p=1$,
%	\begin{align*}
%	w_1(t) = \frac{1}{\sqrt{n}} \bigl\{ \Tr(C_n(t)) - \E[\Tr(C_n(t))]\bigr\} = \frac{1}{\sqrt{n}} \big[n \frac{b_0(t)}{\sqrt n} - \E(n\frac{b_0(t)}{\sqrt n})\big] = b_0(t),
%	\end{align*}
%	as $\E(b_0(t))=0$. So $w_1(t)$ is distributed as $N(0,t)$ and its distribution does not depend on $n$. So we ignore these two cases, for $p=0$ and $p=1$.  
%\end{remark}

%\begin{remark}\label{ft:w_p(t)}
%	For any $p \geq 2 $	and $t>0$, from Theorem \ref{thm:cirmulti} we get that
%	\begin{equation}\label{distributional convergence:w_p(t)}
%	w_p(t) = \frac{1}{\sqrt{n}} \bigl\{ \Tr(C_n(t))^p - \E[\Tr(C_n(t))^p]\bigr\}   \stackrel{\mathcal D}{\rightarrow} N_p(t), 
%	\end{equation}	
%	where $N_p(t)$ has mean 0 and variance $\sigma_{p,t}^2 = (t)^pp!\ \sum_{s=0}^{p-1} f_p(s).$ This convergence of \eqref{distributional convergence:w_p(t)} is also follows from Theorem 1 of 
%	\cite{adhikari_saha2017}, where it was established in total variation norm. 
%\end{remark}
In Section \ref{sec:cov} we prove Theorem \ref{thm:symcircovar}. We derived trace formula and state some results which will be used to prove Theorem \ref{thm:symcircovar}. In Section \ref{sec:poly} we prove Theorem \ref{thm:symcirpoly}. We use  method of moments and  Wick's formula to prove Theorem  \ref{thm:symcirpoly}.

\section{Proof of Theorem \ref{thm:symcircovar}}\label{sec:cov}

%In  this section we prove Theorem \ref{thm:cirpoly} by the method of moments. We have the following corollary for monomials.
%
%\begin{corollary}\label{cor:monomial}
%Let $C_n$ be the circulant matrix with input sequence $\{X_n\}$ satisfying condition \eqref{eqn:condition}. Then for fixed positive integer $d\geq 2$, as $n\to \infty$,
%$$
%\frac{\Tr(C_n^{d})-\E\Tr(C_n^{d})}{\sqrt{n^{d+1}}} \longrightarrow N(0,\sigma_d^2),
%$$
% where $\sigma_d^2=d!\sum_{s=1}^df_d(s)$.
%\end{corollary}
We first define some notation which will be used   in the proof of Theorem \ref{thm:symcirpoly}.
\begin{align} \label{def:A_p_SC}
A_{p}&=\{(j_1,\ldots,j_{p})\suchthat \sum_{i=1}^{p}\epsilon_i j_i=0\; \mbox{(mod n)}, \epsilon_i\in\{+1,-1\}, 1\le j_1,\ldots,j_{p}\le \frac{n}{2}\}, \\
\tilde{A}_{k}&=\{(j_1,\ldots,j_{k})\suchthat \sum_{i=1}^{k}\epsilon_i j_i=0\; \mbox{(mod } \frac{n}{2}) \mbox{ and } \sum_{i=1}^{k}\epsilon_i j_i \neq 0\; \mbox{(mod }n),  \epsilon_i\in\{+1,-1\}, 1\le j_1,\ldots,j_{k}\le \frac{n}{2}\}, \nonumber
\\A_p^{(k)}&=\{(j_1,\ldots,j_p)\in A_p\suchthat j_1+\cdots+j_k - j_{k+1}-\cdots -j_p=0 \;\mbox{ (mod $n$)}\}, \nonumber
\\A_p'^{(k)}&=\{(j_1,\ldots,j_p)\in A_p\suchthat j_1+\cdots+j_k - j_{k+1}-\cdots -j_p=0 \;\mbox{ (mod $n$)}, j_1\neq \cdots \neq j_k \}, \nonumber
\\A_{p,s}^{(k)}&=\{(j_1,\ldots,j_p)\in A_p\suchthat j_1+\cdots+j_k - j_{k+1}-\cdots -j_p=sn\}. \nonumber
\end{align}
In  set $A_p$ and $\tilde{A}_p$, we collect $(j_1,\ldots,j_{p})$ according to their multiplicity. 

      Now we derive a convenient formula of trace for symmetric circulant matrices. First suppose $n$ is  odd positive integers. We write $ n/2$ instead of $\lfloor n/2\rfloor $, as asymptotic is same as $n\to\infty$. Then
\begin{align*}
\Tr(SC_n^p)&=\sum_{\ell=0}^{n-1}\lambda_\ell^p
=\sum_{\ell=0}^{n-1}\left(X_0+2\sum_{j=1}^{n/2}X_j\cos(\omega_\ell j)\right)^p
\\&=\sum_{k=0}^{p}\binom{p}{k}X_0^{p-k}\sum_{\ell=0}^{n-1} \left(\sum_{j=1}^{n/2}X_j(e^{i\omega_\ell j}+e^{-i\omega_\ell j})\right)^{k},
\end{align*}
where $\omega_\ell=\frac{2\pi \ell}{n}$. Since  $\sum_{\ell=0}^{n-1}e^{i\omega_\ell j}=0$ for $j\in \mathbb Z\backslash \{0\}$, we have 
\begin{align}\label{trace formula_SC_odd}
\Tr(SC_n^p)=n\sum_{k=0}^{p}\binom{p}{\ell}X_0^{p-k}\sum_{A_{k}} X_{j_1}X_{j_2}\ldots X_{j_{k}},
\end{align}
where  $A_{k}$ for $k=1,\ldots,p$ is given by
\begin{align*}
A_{k}:=\l\{(j_1,\ldots,j_{k})\suchthat \sum_{i=1}^{k}\epsilon_i j_i=0\; \mbox{(mod n)}, \epsilon_i\in\{+1,-1\}, 1\le j_1,\ldots,j_{k}\le \frac{n}{2}\r\}
\end{align*} 
and $A_0$ is an empty set with the understanding that the contribution from the sum corresponding to $A_0$ is 1. Note, in $A_{k}$, $(j_1,\ldots,j_{k})$ are collected according to their multiplicity. 

  Now suppose $n$ is even positive integers. We write $ n/2$ instead of $ n/2-1$, as asymptotic is same as $n\to\infty$. Then
  \begin{align*}
\Tr(SC_n^p)=\sum_{\ell=0}^{n-1}\lambda_\ell^p
&=\sum_{\ell=0}^{n-1} \Big\{X_0 + (-1)^\ell X_{\frac{n}{2}}+2\sum_{j=1}^{n/2}X_j\cos(\omega_\ell j) \Big\}^p\\
&=\sum_{\ell=0}^{n-1} \sum_{k=0}^{p}\binom{p}{k} {(X_0 + (-1)^\ell X_{\frac{n}{2}})}^{p-k} \Big\{ \sum_{j=1}^{n/2}X_j(e^{i\omega_\ell j}+e^{-i\omega_\ell j}) \Big\}^{k} \\
&=\sum_{k=0}^{p}\binom{p}{k} \Big\{ {(X_0 +  X_{\frac{n}{2}})}^{p-k} \sum_{\ell=0, even}^{n-1} \Big[ \sum_{j=1}^{n/2}X_j (e^{i\omega_\ell j}+e^{-i\omega_\ell j}) \Big] ^{k} \\
& + {(X_0 -  X_{\frac{n}{2}} )}^{p-k} \sum_{\ell=0, odd}^{n-1}  
\Big[ \sum_{j=1}^{n/2}X_j (e^{i\omega_\ell j}+e^{-i\omega_\ell j}) \Big] ^{k} \Big\},
\end{align*}
where $\omega_\ell=\frac{2\pi \ell}{n}$. Since we know
 \begin{align*} 
	\sum_{\ell=0, even}^{n-1}   e^{i\omega_\ell j}
&=  \left\{\begin{array}{ccc} 	 
		 \frac{n}{2} & \text{if}& j =0\; \mbox{(mod } \frac{n}{2}) \\\\
		 	0 & \text{if}&  j \neq0\; \mbox{(mod } \frac{n}{2}),	 
		 	  \end{array}\right.	 	  
\end{align*}
and 
\begin{align*} 
	\sum_{\ell=0, odd}^{n-1}   e^{i\omega_\ell j}
&=  \left\{\begin{array}{ccc} 	 
		 \frac{n}{2} & \text{if}& j =0\; \mbox{(mod } \frac{n}{2}) \mbox{ and }   j = 0\; \mbox{(mod n)} \\\\
		 -\frac{n}{2} & \text{if}& j =0\; \mbox{(mod } \frac{n}{2}) \mbox{ and }   j \neq0\; \mbox{(mod n)} \\\\
		 	0 & \text{if}&  j \neq0\; \mbox{(mod } \frac{n}{2}) ,	 
		 	  \end{array}\right.	 	  
\end{align*} 
Therefore from the above last two observations, $\Tr(SC_n^p)$ will be
\begin{align} \label{trace formula_SC_even}
 \Tr(SC_n^p) &=\frac{n}{2} \sum_{k=0}^{p}\binom{p}{k} \Big[ \Big\{ {(X_0 +  X_{\frac{n}{2}})}^{p-k} + {(X_0 -  X_{\frac{n}{2}})}^{p-k} \Big\} \sum_{A_k} X_{j_1}X_{j_2}\ldots X_{j_{k}} \nonumber\\
& \qquad + \Big\{ {(X_0 +  X_{\frac{n}{2}})}^{p-k} - {(X_0 -  X_{\frac{n}{2}})}^{p-k} \Big\} \sum_{ \tilde{A}_k} X_{j_1}X_{j_2}\ldots X_{j_{k}} \Big] \nonumber \\
%& = \frac{n}{2} \sum_{k=0}^{p}\binom{p}{k} \Big[ Y_k  \sum_{A_k} X_{j_1}X_{j_2}\ldots X_{j_{k}} + \tilde{Y}_k \sum_{ \tilde{A}_k} X_{j_1}X_{j_2}\ldots X_{j_{k}} \Big] 
& = \frac{n}{2} \sum_{k=0}^{p}\binom{p}{k} \Big[ Y_k  \sum_{A_k} X_{J_{k}} + \tilde{Y}_k \sum_{ \tilde{A}_k}  X_{J_{k}} \Big] , \mbox{ say}
\end{align}
where for each  $k=0, 1, 2,\ldots, p$, $A_{k}$ is same as $A_{k}$ of $n$ odd case 
%and the details about $A_{k}$ is given in $n$ odd case 
and $\tilde{A}_{k}$ for $k=1,\ldots,p$ is given by
\begin{align*}
\tilde{A}_{k}:=\l\{(j_1,\ldots,j_{k})\suchthat \sum_{i=1}^{k}\epsilon_i j_i=0\; \mbox{(mod } \frac{n}{2}) \mbox{ and } \sum_{i=1}^{k}\epsilon_i j_i \neq 0\; \mbox{(mod }n),  \epsilon_i\in\{+1,-1\}, 1\le j_1,\ldots,j_{k}\le \frac{n}{2}\r\}.
\end{align*} 
Here note that $\tilde{A}_0$ is an empty set with the understanding that the contribution from the sum corresponding to $\tilde{A}_0$ is $1$ and in $\tilde{A}_k$, $(j_1,\ldots,j_{k})$ are collected according to their multiplicity. Also %note that here
\begin{align} \label{eqn:Y_k}
Y_k & = {(X_0 +  X_{\frac{n}{2}})}^{p-k} + {(X_0 -  X_{\frac{n}{2}})}^{p-k}, \tilde{Y}_k = {(X_0 +  X_{\frac{n}{2}})}^{p-k} - {(X_0 -  X_{\frac{n}{2}})}^{p-k} \\
X_{J_{k}} & = X_{j_1}X_{j_2}\ldots X_{j_{k}}. \nonumber 
\end{align}
From the definition of $A_k$ and $\tilde{A}_k$, observe that $|A_k|= O(n^{k-1})$, because the entries of $A_k$ has one constraint, whereas $|\tilde{A}_k|= O(n^{k-2})$, because the entries of $\tilde{A}_k$ has two constraints. Therefore
\begin{equation} \label{eqn:A,tildeA}
|\tilde{A}_k| < |A_k|.
\end{equation}

The following result will be used in the proof of  Theorem \ref{thm:symcircovar}. 

%For the proof the result, we refer the readers to \cite[Lemma 14]{adhikari_saha2017}.
%number of elements of $B_{2p,s}$, which will be used in the prove of the Theorem \ref{thm:reverseciculant}.
\begin{result} \label{result:def_h}
	Suppose $|A_{p,s}|$ denotes   the cardinality of $A_{p,s}$. Then 
	$$
	h_p(k):=
\lim_{n\to \infty}\frac{|A_p^{(k)}|}{n^{p-1}}= \frac{1}{(p-1)!}\sum_{s=-\lceil\frac{p-k}{2}\rceil}^{\lfloor\frac{k}{2}\rfloor}\sum_{q=0}^{2s+p-k}(-1)^q\binom{p}{q}\l(\frac{2s+p-k-q}{2}\r)^{p-1}, 
$$
where $\lceil x\rceil$ denotes the smallest integer not less than $x$.

\end{result}
For the proof of Result \ref{result:def_h}, we refer to \cite[Lemma 14]{adhikari_saha2017}. 
%We will use this result to prove Theorem \ref{thm:symcircovar}.
Now for a given vector $(j_1, j_2, \ldots, j_p)$, we define a term called {\it opposite sign pair matched} elements of the vector. 
\begin{definition}\label{def:odd-even}
	Suppose $(j_1, j_2, \ldots, j_p) \in A_p$.
	% and $j_k, j_\ell \in (j_1, j_2, \ldots, j_p)$ is a pair. 
	 We say $j_k, j_\ell$ is {\it opposite sign pair matched}, if $\epsilon_k$ and $\epsilon_\ell$ corresponding to $j_k$ and $j_\ell$, respectively, are of opposite sign and $j_k= j_\ell$, where $\epsilon_k$ and $\epsilon_\ell$ are corresponds to  (\ref{def:A_p_SC}).
%	 one of the elements of the pair have positive sign and other one have negative sign. 
	 For example; In $(2,3,5,2)$, entry $2$ is {\it opposite sign pair matched}, if $\epsilon_1=1$ and $\epsilon_4=-1$ or $\epsilon_1=-1$ and $\epsilon_4=1$ whereas if $\epsilon_1$ = $\epsilon_4= 1$ or $\epsilon_1$ = $\epsilon_4= -1$, then $2$ is not  {\it opposite sign pair matched}. Similarly, we can also define {\it opposite sign pair matched} elements of $\tilde{A}_p$. We shall call, vector $(j_1, j_2, \ldots, j_p)$ is {\it opposite sign pair matched}, if all the entries of $(j_1, j_2, \ldots, j_p)$ are {\it opposite sign pair matched}.
\end{definition}
 Observe that, if $(j_1, j_2, \ldots, j_p)\in A_{p}$, that is, $\sum_{i=1}^{p}\epsilon_i j_i=0 \mbox{ (mod $n$) }$ and each entry of $\{j_1, j_2, \ldots, j_p\}$ has multiplicity greater than or equal to two. Then the maximum number of free variable in $(j_1, j_2, \ldots, j_p)$ will be $\frac{p}{2}$ only when $p$ is even and $(j_1, j_2, \ldots, j_p)$ is {\it opposite sign pair matched}. We shall use this observation in proof of Theorem \ref{thm:symcircovar}, for maximum contribution.
%\begin{lemma}\label{lem:w_p}
%	Suppose $p, q\geq 1$.
%	Then
%	\begin{align}\label{eqn:covarianceRC}
%	\Cov\big(w_p,w_q\big)&=\frac{1}{n^{p+q-1}} \sum_{A_{2p}, A_{2q}} \Big\{\E[x_{i_1}x_{i_2}\cdots x_{i_{2p}} x_{j_1} x_{j_2}\cdots x_{j_{2q}} ]\nonumber \\
%	&\qquad -\E[ x_{i_1} x_{i_2}\cdots x_{i_{2p}} ] \E[x_{j_1} x_{j_2}\cdots x_{j_{2q}}  ]        \Big\},
%	\end{align}
%	where $w_p$ is as defined in (\ref{eqn:RCw_p}).

Now assuming the above Result, we proceed to prove Theorem \ref{thm:symcircovar}. 
%To prove Theorem \ref{thm:symcircovar}, we will use similar idea as [ saha and maurya, Theorem 1] and  \cite[Theorem 5]{adhikari_saha2017}. 
We shall use trace formula of $SC_n$ to prove \ref{thm:symcircovar}. Since for odd and even value of $n$, we have different trace formula, therefore we shall prove \ref{thm:symcircovar} in two steps. In Step 1, we calculate limit of $\Cov\big(w_p,w_q\big)$ as $n\to\infty$ with odd $n$ and in Step 2, we calculate limit of $\Cov\big(w_p,w_q\big)$ as $n\to\infty$ with even $n$. We shall show that for both the cases, even and odd value of $n$,  limit of $\Cov\big(w_p,w_q\big)$ is same.
\begin{proof}[Proof of Theorem \ref{thm:symcircovar}]
	Since $\E(w_p)=\E(w_q)=0$, therefore we get
	\begin{align*}
	\Cov\big(w_p,w_q\big)&=\E[w_p w_q]
	= \frac{1}{n} \Big\{ \E[\Tr(SC_n)^{p}\Tr(SC_n)^{q} ]- \E[\Tr(SC_n)^{p}]\E[\Tr(SC_n)^{q}]   \Big\}. 
	\end{align*}
	First we suppose $\Cov\big(w_p,w_q\big)$ for odd value of $n$. \\\
	
	\noindent \textbf{Step 1.} Suppose $n$ is odd, then by the trace formula (\ref{trace formula_SC_odd}), we get
	\begin{align*}
	\E[\Tr(SC_n)^{p}]&= \E\Big[n\sum_{k=0}^{p}\binom{p}{k}X_0^{p-k}\sum_{A_{k}}\frac{X_{i_1}}{ \sqrt{n}} \cdots \frac{X_{i_{k}}}{ \sqrt{n}} \Big]
	= \frac{1}{ n^ {\frac{p}{2}-1}} \E\Big[ \sum_{k=0}^{p}\binom{p}{k}X_0^{p-k}\sum_{A_{k}} X_{i_1} \cdots X_{i_{k}}\Big].	
	\end{align*}
	Therefore
	\begin{align} \label{eqn:T_1+T_2_SC}
	\Cov\big(w_p,w_q\big)&=\E[w_p w_q] \nonumber\\
	&= \frac{1}{n^{\frac{p+q}{2}-1}} \Big[ \E\Big\{ \Big( \sum_{k=0}^{p}\binom{p}{k} X_0^{p-k}\sum_{A_{k}}  X_{i_1} \cdots X_{i_{k}} \Big) \Big( \sum_{\ell=0}^{q}\binom{q}{\ell}X_0^{q-\ell}\sum_{A_{\ell}} X_{j_1} \cdots X_{j_{\ell}} \Big)  \Big\}\nonumber \\
	 &\qquad  - \E \Big( \sum_{k=0}^{p}\binom{p}{k}X_0^{p-k}\sum_{A_{k}}  X_{i_1} \cdots X_{i_{k}} \Big) \E \Big( \sum_{\ell=0}^{q}\binom{q}{\ell}X_0^{q-\ell}\sum_{A_{\ell}} X_{j_1} \cdots X_{j_{\ell}} \Big)   \Big] \nonumber\\
%	 &=\frac{1}{n^{\frac{p+q}{2}-1}} \sum_{k, \ell =0}^{p, q}\binom{p}{k} \binom{q}{\ell}  \sum_{A_{k}, A_{\ell}} \Big\{\E[ X_0^{p+q-k-\ell} X_{i_1} \cdots X_{i_{k}} X_{j_1} X_{j_2}\cdots X_{j_{q}} ]\nonumber \\
%	 &\qquad -\E[X_0^{p-k} X_{i_1} \cdots X_{i_{k}} ] \E[X_0^{q-\ell} X_{j_1} \cdots X_{j_{\ell}} ]        \Big\} \nonumber \\
	  &=\frac{1}{n^{\frac{p+q}{2}-1}} \sum_{k, \ell =0}^{p, q}\binom{p}{k} \binom{q}{\ell}  \sum_{A_{k}, A_{\ell}} \Big\{\E[ X_0^{p+q-k-\ell}] \E[X_{i_1} \cdots X_{i_{k}} X_{j_1} X_{j_2}\cdots X_{j_{\ell}} ]\\
	 &\qquad -\E[X_0^{p-k}]\E[ X_{i_1} \cdots X_{i_{k}} ] \E[X_0^{q-\ell}]\E[ X_{j_1} \cdots X_{j_{\ell}} ]   \Big\}.	 \nonumber 
	\end{align}
	Depending on the values of $k$ and $\ell$, the following two cases arise. \\\
	
	\noindent \textbf{Case I.} \textbf{Either $k=p, \ell \leq q$ or $\ell=q, k \leq p$ :}  Since in this case, we always get $\E[ X_0^{p+q-k-\ell}]= \E[X_0^{p-k}] \E[X_0^{q-\ell}]$. Therefore, if $\{i_1,i_2,\ldots,i_{k}\}\cap \{j_1,j_2,\ldots,j_{\ell}\}=\emptyset$ then from independence of $X_i$'s, we get
	$$\E[ X_0^{p+q-k-\ell}] \E[X_{i_1} \cdots X_{i_{k}} X_{j_1} X_{j_2}\cdots X_{j_{\ell}} ] -\E[X_0^{p-k}]\E[ X_{i_1} \cdots X_{i_{k}} ] \E[X_0^{q-\ell}]\E[ X_{j_1} \cdots X_{j_{\ell}} ]  =0.$$
	%the contribution from the right hand side is zero. 
	Hence in this case, we can get non-zero contribution from (\ref{eqn:T_1+T_2_SC}) only when there is at least one cross-matching among $\{i_1,\ldots,i_{k}\}$ and $\{j_1,\ldots,j_{\ell}\}$, i.e., $\{i_1,i_2,\ldots,i_{k}\}\cap \{j_1,j_2,\ldots,j_{\ell}\}\neq\emptyset$ for some $k =0, 1, \ldots, p$ and $\ell=0, 1, \ldots, q$. 
	So from the above observation, (\ref{eqn:T_1+T_2_SC}) can be written as
	\begin{align} \label{eqn:T_k,l_SC}
	\lim_{n\to\infty}  \Cov\big(w_p,w_q\big)  &=\lim_{n\to\infty}  \frac{1}{n^{\frac{p+q}{2}-1}} \sum_{k, \ell =0}^{p, q}\binom{p}{k} \binom{q}{\ell}  \sum_{m=1}^{ \min \{k,\ell\} } \sum_{I_m} \Big\{\E[ X_0^{p+q-k-\ell}] \E[X_{i_1} \cdots X_{i_{k}} X_{j_1} \cdots X_{j_{\ell}} ]\nonumber \\
	 &\qquad -\E[X_0^{p-k}]\E[ X_{i_1} \cdots X_{i_{k}} ] \E[X_0^{q-\ell}]\E[ X_{j_1} \cdots X_{j_{\ell}} ]        \Big\}	 \nonumber \\	
	& =\lim_{n\to\infty}  \frac{1}{n^{\frac{p+q}{2}-1}} \sum_{k, \ell =0}^{p, q}\binom{p}{k} \binom{q}{\ell} \sum_{m=1}^{ \min \{k,\ell\} } T^m_{k,\ell}, \mbox{ say},
\end{align}		 
	where for each $m=1,2, \ldots, \min\{p,q\}$, $I_m$ is defined as
	\begin{equation}\label{def:I_k}
	I_{m}:=\{((i_1,\ldots,i_{k}),(j_1,\ldots,j_{\ell}))\in A_{k}\times A_{\ell}\suchthat |\{i_1,\ldots,i_{k}\}\cap \{j_1,\ldots,j_{\ell}\}|=m\}.
	\end{equation}
	Now we calculate the contribution due to the typical term $T^m_{k,\ell}$ of (\ref{eqn:T_k,l_SC}) for some fixed value of $k= 1, 2, \ldots, p$, $\ell= 1, 2, \ldots, q$ and $m =1,2, \ldots, \min\{k,\ell\}$.  Since from (\ref{eqn:condition}), we have 
%	$$ \frac{1}{n^{\frac{p+q}{2}-1}} \sum_{A_{k}, A_{\ell}} \E[X_{i_1}\ldots X_{i_{k}}X_{j_1}\ldots X_{j_{\ell}}]-\E[X_{i_1}\ldots X_{i_{k}}]\E[X_{j_1}\ldots X_{j_{\ell}}] $$ 
	\begin{equation*}
	\E[X_i]= 0, \ \E(X^2_i)=1 \mbox{ and } \sup_{i \geq 1}\E(|X_i|^{k})= \alpha_k < \infty \mbox{ for } k \geq 3.
	\end{equation*}  
	Therefore there exist $\gamma >0$, which depends only on $k$ and $\ell$, such that 
	\begin{align} \label{eqn:gamma_1}
%	\sum_{A_{k}, A_{\ell}} \E[ X_0^{p+q-k-\ell}] \E[X_{i_1}\ldots X_{i_{k}}X_{j_1}\ldots X_{j_{\ell}}]-\E[ X_0^{p-k}]\E[ X_0^{q-\ell}]  \E[X_{i_1}\ldots X_{i_{k}}]\E[X_{j_1}\ldots X_{j_{\ell}}]
	|T^m_{k,\ell}| & = \sum_{I_m} |\E[ X_0^{p+q-k-\ell} X_{i_1} \cdots X_{i_{k}} X_{j_1} \cdots X_{j_{\ell}} ] -\E[X_0^{p-k} X_{i_1} \cdots X_{i_{k}} ] \E[X_0^{q-\ell}  X_{j_1} \cdots X_{j_{\ell}} ] | \nonumber \\   
	& \leq \gamma |B_{k,\ell}|,
	\end{align}
	where  $B_{k,\ell} \subseteq A_k \times A_\ell$ with conditions that  $\{ i_1, i_2,\ldots,i_k \} \cap \{ j_1, j_2,\ldots,j_\ell \} \neq \emptyset$ and each element of set $\{ i_1, i_2,\ldots,i_{k}\} \cup \{ j_1, j_2, \ldots, j_\ell\} $ has multiplicity greater than or equal to two. So, to solve (\ref{eqn:gamma_1}), it is enough to calculate the cardinality of $B_{k,\ell}$. Suppose $((i_1, i_2,\ldots,i_{k}), (j_1,j_2,\ldots,j_{\ell}))\in B_{k,\ell}$ with $|\{i_1,\ldots,i_{k}\}\cap \{j_1,\ldots,j_{\ell}\}|=m$, for some $m=1,2,\ldots, \min\{k,\ell\}$,  
	%without loss of generality, we suppose $p\leq q$, that is, $k=1,2,\ldots, p$,
	 where $|\{\cdot\}|$ denotes cardinality of the set  $\{\cdot\}$. Therefore typical element of $B_{k,\ell}$ will look like $$((d_1, d_2,\ldots, d_m, i_{m+1}, \ldots, i_{k}), (d_1,d_2,\ldots, d_m, j_{m+1}, \ldots, j_{\ell})).$$
	Observe that, we shall get maximum number of free entries in $B_{k,\ell}$, if following conditions hold
	
	\begin{enumerate}
		\item [(i)] each elements of $ \{d_1, d_2, \ldots, d_{m}\}$ are distinct,
		\item [(ii)] if $k-m$ is even. Then 
%		for vector $(d_1,\ldots, d_{m}, i_{m+1},\ldots,i_{k})$, $\sum_{t=m+1}^{k} \epsilon_t = 0$ with 
		$(i_{m+1}, \ldots, i_{k})$ is {\it opposite sign pair matched}
%		for vector $(d_1,\ldots, d_{m}, j_{m+1},\ldots, j_{\ell})$, $\sum_{t=m+1}^{\ell} \epsilon_t = 0$ with $(j_{m+1}, \ldots, j_{\ell})$ is {\it opposite sign pair matched} 
with $ \{d_1, d_2, \ldots, d_{m}\} \cap \{i_{m+1},\ldots,i_{k}\}=\emptyset$. Similar condition also hold when $\ell-m$ is even,
		\item [(iii)] if $k-m$ is odd. Then
%		Then there exist $i^{*} \in (i_{m+1},\ldots,i_{k})$ such that $i^{*}$ equal to $d_s$ for some $s=1,2, \ldots, m$. Without loss of generality, we suppose $i^{*}$ is {\it opposite sign pair matched} with $d_s$ and $d_s= d_{m}$. Similarly, there exist $j^{*} \in (j_{m},\ldots,j_{\ell})$ which is {\it opposite sign pair matched} with $d_{m-1}$. such that $ \{d_1, d_2, \ldots, d_{m-2}\} \cap \{i_{m+1},\ldots,i_{k}\} \setminus \{i^*\}\cap \{j_{m+1},\ldots,j_{\ell}\} \setminus \{j^*\}=\emptyset$ and $\{i_{m+1},\ldots,i_{k}\} \setminus \{i^*\}, \{j_{m+1},\ldots,j_{\ell}\} \setminus \{j^*\}$ are {\it opposite sign pair matched}.
$\{i_{m+1},\ldots,i_{k}\} \setminus \{i^*\}$ is {\it opposite sign pair matched} and
 $ \{d_1, d_2, \ldots, d_{m-2}\} \cap \{i_{m+1},\ldots,i_{k}\} \setminus \{i^*\}=\emptyset$, where $i^{*}$ is {\it opposite sign pair matched} with $d_s$ for some $s=1,2, \ldots, m$.
 %, without loss of generality, we have supposed that $d_s=d_m$. 
 Similar condition also hold when $\ell-m$ is odd.
	 \end{enumerate}
%	 \textcolor{red}{OR}
%	 \begin{enumerate}
%		\item [(i)] each elements of $ \{d_1, d_2, \ldots, d_{m}\}$ are distinct,
%		\item [(ii)] $ \{d_1, d_2, \ldots, d_{m}\} \cap \{i_{m+1},\ldots,i_{k}\}\cap \{j_{m+1},\ldots,j_{\ell}\}=\emptyset$,
%%		For vector $(d_1,\ldots, d_{m}, i_{m+1},\ldots,i_{k})$, $\sum_{t=m+1}^{k} \epsilon_t = 0$ and $(i_{m+1}, \ldots, i_{k})$ is {\it opposite sign pair matched},
%		\item [(iii)] $(i_{m+1}, \ldots, i_{k})$
%%		For vector $(d_1,\ldots, d_{m}, j_{m+1},\ldots, j_{\ell})$, $\sum_{t=m+1}^{\ell} \epsilon_t = 0$ 
%		and $(j_{m+1}, \ldots, j_{\ell})$ are {\it opposite sign pair matched}.
%	 \end{enumerate}
%	Due to the above consideration (ii) and (iii), the constraints, $\sum_{t=1}^{m}\epsilon_t d_t +\sum_{t=m+1}^{k}\epsilon_t i_t =0\; \mbox{(mod n})$ and $\sum_{t=1}^{m}\epsilon_t d_t  + \sum_{t=m+1}^{\ell}\epsilon_t j_t =0\; \mbox{(mod n})$ will change into one constraint $\sum_{t=1}^{m}\epsilon_t d_t =0\; \mbox{(mod n})$. 
Under the above assumption,  the cardinality of $B_{k,\ell}$ will be
	
	 \begin{align}\label{card_B_k_l}
	 |B_{k,\ell}|  & = \left\{\begin{array}{ll} 	 
		 	 O(n^{m-1+\frac{k-m}{2} +\frac{\ell -m}{2}}) & \text{if}\ (k-m) \mbox{ and } (\ell-m) \mbox{ both are even,}\\\\
		 	  O(n^{m-3+\frac{k-m+1}{2} +\frac{\ell -m+1}{2}}) & \text{if}\ (k-m) \mbox{ and } (\ell-m) \mbox{ both are odd,}\\\\
			O(n^{m-2+\frac{k-m+\ell -m+1}{2}})& \text{otherwise}, 	 	 
		 	  \end{array}\right.	 \nonumber \\
	 & = \left\{\begin{array}{ll} 	 
		 	 O(n^{\frac{k+\ell}{2}-1}) & \text{if}\ (k-m) \mbox{ and } (\ell-m) \mbox{ both are even,}\\\\
			o(n^{\frac{k+\ell}{2}-1}) & \text{otherwise}. 	 	 
		 	  \end{array}\right.		 
	 \end{align}
%	where $[x]$ denotes the greatest integer less than or equal to $x$. 
Now from (\ref{eqn:gamma_1}) and (\ref{card_B_k_l}), we get
	\begin{align} \label{card_T_k,l}
	|T^m_{k,\ell}|  & = \left\{\begin{array}{ll} 	 
		 	 O(n^{\frac{k+\ell}{2}-1}) & \text{if}\ k, \ell \mbox{ and } m \mbox{ all are even or }  k, \ell \mbox{ and } m \mbox{ all are odd},\\\\
			o(n^{\frac{k+\ell}{2}-1}) & \text{otherwise}. 	 	 
		 	  \end{array}\right.	
	\end{align}
	On using (\ref{eqn:T_k,l_SC}) and (\ref{card_T_k,l}), we get that $T^m_{k,\ell}$ has non-zero contribution in (\ref{eqn:T_k,l_SC}) only when $k=p$ and $\ell=q$. In fact $T^m_{k,\ell}$ has non-zero contribution only when either $p,q, m$ all are even or $p,q, m$ all are odd. So, if we use (\ref{card_T_k,l}) in (\ref{eqn:T_k,l_SC}), we get 
	\begin{align} \label{eqn:lim_cov_SC}
	&\lim_{n\to\infty}  \Cov\big(w_p,w_q\big) \nonumber \\
	&= \displaystyle\lim_{n\to\infty} \frac{1}{n^{\frac{p+q}{2}-1}} \sum_{m=1}^{ \min \{p,q\} } \sum_{I_m} \Big\{\E[X_{i_1} \cdots X_{i_{p}} X_{j_1} \cdots X_{j_{q}} ] -\E[ X_{i_1} \cdots X_{i_{p}} ] \E[ X_{j_1} \cdots X_{j_{q}} ]  \Big\}	\nonumber \\
	&=  \left\{\begin{array}{ll} 	 
	        \displaystyle\lim_{n\to\infty} \frac{1}{n^{\frac{p+q}{2}-1}} \sum_{r=1}^{ \min\{ \frac{p}{2},\frac{q}{2} \} } \sum_{I_{2r}} \big( \E[X_{i_1} \cdots X_{i_{p}} X_{j_1} \cdots X_{j_{q}} ]\\
	\quad -\E[ X_{i_1} \cdots X_{i_{p}} ] \E[ X_{j_1} \cdots X_{j_{q}} ]  \big) & \text{if}\ p,  q \mbox{ both are even,}\\\\
		 	 \displaystyle\lim_{n\to\infty}\frac{1}{n^{\frac{p+q}{2}-1}} \sum_{r=0}^{ \min\{ \frac{p-1}{2},\frac{q-1}{2} \} } \sum_{I_{2r+1}} \big( \E[X_{i_1} \cdots X_{i_{p}} X_{j_1} \cdots X_{j_{q}} ]\\
	\qquad -\E[ X_{i_1} \cdots X_{i_{p}} ] \E[ X_{j_1} \cdots X_{j_{q}} ]  \big) & \text{if}\ p,  q \mbox{ both are odd,}\\\\
			0 &  \text{otherwise}. 	 	 
		 	  \end{array}\right.			 	  
	\end{align}
%	where for even $m$, $m=2r$ and for odd $m$, $m=2r+1$, and $I_m$ is as defined in (\ref{def:I_k}).
%	 \begin{equation}\label{def:I_k}
%	I_{m}:=\{((i_1,\ldots,i_{p}),(j_1,\ldots,j_{q}))\in A_{p}\times A_{q}\suchthat |\{i_1,\ldots,i_{p}\}\cap \{j_1,\ldots,j_{q}\}|=m\}.
%	\end{equation}
Now we calculate right hand side of (\ref{eqn:lim_cov_SC}). Depending on values of $p,q$, following two subcases arise.\\\\
%	Since right hand side of (\ref{eqn:lim_cov_SC}) is non-zero for even values of $p, q$ and odd values of $p,q$. Therefore the following two subcases arise.\\\\
	\noindent \textbf{subcase (i)} \textbf{$p,q$ both are even:} 
	First recall, the typical term of $I_{2r}$ is 
	$$((d_1, d_2,\ldots, d_{2r}, i_{2r+1}, \ldots, i_{p}), (d_1,d_2,\ldots, d_{2r}, j_{2r+1}, \ldots, j_{q})).$$
	For such an element of $I_{2r}$, the number of free entries in $I_{2r}$ will be maximum, if following conditions hold
	\begin{enumerate}
	% \item [(i)] each entries of $ \{d_1, d_2, \ldots, d_{2r}\}$ are distinct,
		\item [(i)] $ \{d_1, d_2, \ldots, d_{2r}\} \cap \{i_{2r+1},\ldots,i_{p}\}\cap \{j_{2r+1},\ldots,j_{q}\}=\emptyset$, 
		\item [(ii)] $(i_{2r+1}, \ldots, i_{p})$ and $(j_{2r+1}, \ldots, j_{q})$ are {\it opposite sign pair matched}.
	\end{enumerate}
	Due to the above consideration, the constraints, $\sum_{t=1}^{2r}\epsilon_t d_t +\sum_{t=2r+1}^{p}\epsilon_t i_t =0\; \mbox{(mod n})$ and $\sum_{t=1}^{2r}\epsilon_t d_t  + \sum_{t=2r+1}^{q}\epsilon_t j_t =0\; \mbox{(mod n})$ will change into one constraint 
\begin{equation}\label{eqn:constraint_d_even}
\sum_{t=1}^{2r}\epsilon_t d_t =0\; \mbox{(mod n}).
\end{equation}	
Now first we consider $r\geq 2$, later we shall deal $r=1$ case.
	Note that for $r=2,3, \ldots, \min\{p,q\}$, if we assume each entries of $ \{d_1, d_2, \ldots, d_{2r}\}$ are distinct, then cardinality of $I_{2r}$ will be of the order $O(n^{2r-1 +\frac{p-2r}{2} +\frac{q-2r}{2} })= O(n^{\frac{p+q}{2} -1}),$ where $(-1)$ is arising due to (\ref{eqn:constraint_d_even}).  In any other situation,
	% any one of the above conditions does not holds, 
	cardinality of $I_{2r}$ will be $o(n^{p+q-1}).$ Also note that, as each entries of $ \{d_1, d_2, \ldots, d_{2r}\}$ are distinct, therefore
	$$ \E[ X_{i_1} \cdots X_{i_{p}} ] \E[ X_{j_1} \cdots X_{j_{q}} ] =0. $$
	Hence for each fixed $r\geq 2$, first part of (\ref{eqn:lim_cov_SC}) ($p,q$ both even) will be 
	\begin{align} \label{eqn:I_2k2_SC}
%	\lim_{n\to\infty} &\Cov\big(w_p,w_q\big)=
%  &\lim_{n\to\infty}  \frac{1}{n^{\frac{p+q}{2}-1}} \sum_{I_{2r}}  \E[X_{i_1} \cdots X_{i_{p}} X_{j_1} \cdots X_{j_{q}} ] -\E[ X_{i_1} \cdots X_{i_{p}} ] \E[ X_{j_1} \cdots X_{j_{q}} ]  \nonumber \\
	&  \lim_{n\to\infty}  \frac{1}{n^{\frac{p+q}{2}-1}} \sum_{I_{2r}}  \E[X_{i_1} \cdots X_{i_{p}} X_{j_1} \cdots X_{j_{q}} ] \nonumber \\
	& = \lim_{n\to\infty} \frac{1}{n^{\frac{p+q}{2}-1}}  a_r  (n/2)^{\frac{p-2r}{2}+\frac{q-2r}{2}} \sum_{A_{2r}, A_{2r}}  \E[X_{i_1} \cdots X_{i_{2r}} X_{j_1} \cdots X_{j_{2r}} ] \nonumber \\
	& =  \frac{a_r}{2^{\frac{p+q-4r}{2}}  } \lim_{n\to\infty} \frac{1}{n^{2r-1}}  \sum_{A_{2r}, A_{2r}}  \E[X_{i_1} \cdots X_{i_{2r}} X_{j_1} \cdots X_{j_{2r}} ],
	\end{align}
	where $ a_r=	\binom{p}{p-2r}\binom{p-2r}{\frac{p-2r}{2}} (\frac{p-2r}{2})!\binom{q}{q-2r}\binom{q-2r}{\frac{q-2r}{2}} (\frac{q-2r}{2})!$.
%\begin{equation}\label{eqn:a_r_even}
% a_r=	\binom{p}{p-2r}\binom{p-2r}{\frac{p-2r}{2}} (\frac{p-2r}{2})!\binom{q}{q-2r}\binom{q-2r}{\frac{q-2r}{2}} (\frac{q-2r}{2})!.
%\end{equation}
 $a_r$ factor is arising for pair-matching of $(p-2r)$ many variables in $(i_1,i_2,\ldots,i_{p})$ and  $(j_1,j_2,\ldots, j_{q})$ both  with opposite sign. In $(i_1,i_2,\ldots,i_{p})$, we can choose $(p-2r)$ variables in $\binom{p}{p-2r}$ many ways. Out of $(p-2r)$ variables, $(\frac{p-2r}{2})$ many variables  can be chosen with positive sign in $\binom{p-2r}{\frac{p-2r}{2}}$ many ways. After free choice of $(\frac{p-2r}{2})$ variables with positive sign, rest of the $(\frac{p-2r}{2})$ variables with negative sign can be chosen in $(\frac{p-2r}{2})!$ ways. Therefore for pair matching of $(p-2r)$ many variables  in $(i_1,i_2,\ldots,i_{p})$ with opposite sign, we get $(\binom{p}{p-2r}\binom{p-2r}{\frac{p-2r}{2}} (\frac{p-2r}{2})!)$ factor. Similarly from $(j_1,j_2,\ldots,j_{q})$, we get $(\binom{q}{q-2r}\binom{q-2r}{\frac{q-2r}{2}} (\frac{q-2r}{2})!)$ factor.  %Also note $B_1'$ is an empty set. Therefore the term corresponding to $k=0$ in \eqref{eqn:sc2.6} is zero.   
Now from \eqref{eqn:I_2k2_SC}, we get

\begin{align*}
\lim_{n\to\infty}  \frac{1}{n^{\frac{p+q}{2}-1}} \sum_{I_{2r}}  \E[X_{i_1} \cdots X_{i_{p}} X_{j_1} \cdots X_{j_{q}} ]
& = \frac{a_r}{2^{\frac{p+q-4r}{2}}  } \lim_{n\to\infty} \frac{1}{n^{2r-1}}  \sum_{A_{2r}, A_{2r}}  \E[X_{i_1} \cdots X_{i_{2r}} X_{j_1} \cdots X_{j_{2r}} ] \\
&=\frac{a_r}{2^{\frac{p+q-4r}{2}}  } \lim_{n\to \infty}\frac{1}{n^{2r-1}}\sum_{s=0}^{2r}\binom{2r}{s}^2 s!(2r-s)! |A_{2r}'^{(s)}|\\
&= \frac{a_r}{2^{\frac{p+q-4r}{2}}  } \sum_{s=0}^{2r}\binom{2r}{s}^2 s!(2r-s)! \lim_{n\to \infty}\frac{|A_{2r}^{(s)}|}{n^{2r-1}}. 
\end{align*}
The factor $\binom{2r}{s}^2$ appeared because in $\binom{2r}{s}$ ways we can choose $s$ many $+1$ from $\{\epsilon_1,\ldots,\epsilon_{2r}\}$ in one $A_{2r}'$. The factor $(s!(2r-s)!)$ appeared because for each choice of $(i_1,\ldots,i_{2r})$ we have  $(s!(2r-s)!)$ many choice for $(j_1,\ldots,j_{2r})$. 
Now using Result \ref{result:def_h} in right side of the last above equality, we get
\begin{align}\label{eqn:p,q_even1}
\lim_{n\to\infty}  \frac{1}{n^{\frac{p+q}{2}-1}} \sum_{I_{2r}}  \E[X_{i_1} \cdots X_{i_{p}} X_{j_1} \cdots X_{j_{q}} ]
&= \frac{a_r}{2^{\frac{p+q-4r}{2}}  } \sum_{s=0}^{2r}\binom{2r}{s}^2 s!(2r-s)! \ h_{2r}(s),
%\frac{1}{(2k)!}\sum_{s=-\lceil\frac{2k+1-\ell}{2}\rceil}^{\lfloor\frac{\ell}{2}\rfloor}\sum_{q=0}^{2s+2k+1-\ell}\binom{2k+1}{q}\l(\frac{2s+2k+1-\ell-q}{2}\r)^{2k}.\nonumber
\end{align}	
where $h_{2r}(s)$ is defined in Result \ref{result:def_h}.

Now we calculate first part of (\ref{eqn:lim_cov_SC}) ($p,q$ both even) for $r=1$.
	Note that, if $r=1$ in this Case, then from (\ref{eqn:constraint_d_even}), we get $d_1=d_2$, and hence 
	\begin{align}\label{eqn:p,q_even2}
\lim_{n\to\infty}  \frac{1}{n^{\frac{p+q}{2}-1}} \sum_{I_{2}}  \E[X_{i_1} \cdots X_{i_{p}} X_{j_1} \cdots X_{j_{q}} ]- \E[ X_{i_1} \cdots X_{i_{p}} ] \E[ X_{j_1} \cdots X_{j_{q}} ]
&= \frac{a_1}{2^{\frac{p+q-4}{2}}  } (\E X^4_1- (\E X^2_1)^2) \nonumber \\
&= \frac{a_1}{2^{\frac{p+q-4}{2}}  } (\E X^4_1- 1),
\end{align}
where $a_1=(\binom{p}{p-2}\binom{p-2}{\frac{p-2}{2}} (\frac{p-2}{2})!)(\binom{q}{q-2}\binom{q-2}{\frac{q-2}{2}} (\frac{q-2}{2})!)$. Therefore from (\ref{eqn:lim_cov_SC}), (\ref{eqn:p,q_even1}) and (\ref{eqn:p,q_even2}), we get
\begin{align} \label{eqn:p,q_even}
  \lim_{n\to\infty} &\frac{1}{n^{\frac{p+q}{2}-1}} \sum_{r=1}^{ \min\{ \frac{p}{2},\frac{q}{2} \} } \sum_{I_{2r}} \big( \E[X_{i_1} \cdots X_{i_{p}} X_{j_1} \cdots X_{j_{q}} ] -\E[ X_{i_1} \cdots X_{i_{p}} ] \E[ X_{j_1} \cdots X_{j_{q}} ]  \big)  \\
  &= \frac{a_1}{2^{\frac{p+q-4}{2}}  } (\E X^4_1- 1) + \sum_{r=2}^{ \min\{ \frac{p}{2},\frac{q}{2} \} } \frac{a_r}{2^{\frac{p+q-4r}{2}}  } \sum_{s=0}^{2r}\binom{2r}{s}^2 s!(2r-s)! \ h_{2r}(s).\nonumber
\end{align}
	
	\noindent \textbf{subcase II.} \textbf{$p,q$ both are odd:} In this case we calculate right hand side of (\ref{eqn:lim_cov_SC}) for odd value of $p, q$ and $m$. If $m=2r+1$ for $r=0,1, \ldots, \min\{\frac{p-1}{2}, \frac{q-1}{2}\}$, then the typical term of $I_{2r+1}$ looks like
	$$((d_1, d_2,\ldots, d_{2r+1}, i_{2r+2}, \ldots, i_{p}), (d_1,d_2,\ldots, d_{2r+1}, j_{2r+2}, \ldots, j_{q}))$$
	and for such an elements of $I_{2r+1}$, the number of free entries in $I_{2r+1}$ will be maximum, if following conditions hold
	\begin{enumerate}
	\item [(i)] each entries of $ \{d_1, d_2, \ldots, d_{2r+1}\}$ are distinct,
		\item [(ii)] $ \{d_1, d_2, \ldots, d_{2r+1}\} \cap \{i_{2r+2},\ldots,i_{p}\}\cap \{j_{2r+2},\ldots,j_{q}\}=\emptyset$,
%		For vector $(d_1,\ldots, d_{2r+1}, i_{2r+2},\ldots,i_{p})$, $\sum_{t=2r+2}^{p} \epsilon_t = 0$ and $(i_{2r+2}, \ldots, i_{p})$ is {\it opposite sign pair matched},
		\item [(iii)] $(i_{2r+2}, \ldots, i_{p})$ and $(j_{2r+2}, \ldots, j_{q})$ are {\it opposite sign pair matched},
	\end{enumerate}
%	Due to the above considerations, the constraints, $\sum_{t=1}^{2r+1}\epsilon_t d_t +\sum_{t=2r+2}^{p}\epsilon_t i_t =0\; \mbox{(mod n})$ and $\sum_{t=1}^{2r+1}\epsilon_t d_t  + \sum_{t=2r+2}^{q}\epsilon_t j_t =0\; \mbox{(mod n})$ will change into one constraint 
%\begin{equation*}\label{eqn:constraint_d}
%\sum_{t=1}^{2r+1}\epsilon_t d_t =0\; \mbox{(mod n}).
%\end{equation*}	
%  
%  Hence for any $r=0, 1, 2, \ldots, \min\{\frac{p-1}{2}, \frac{q-1}{2}\}$, the cardinality of $I_{2r+1}$ 
  and the contribution will be of the order $O(n^{2r+1-1 +\frac{p-2r-1}{2} +\frac{q-2r-1}{2} })= O(n^{\frac{p+q}{2} -1})$, where $(-1)$ is arising due to the constraint, $\sum_{t=1}^{2r+1}\epsilon_t d_t =0\; \mbox{(mod n}).$ In any other situation, the cardinality of $I_{2r+1}$ will be $o(n^{p+q-1}).$ Since each entries of $ \{d_1, d_2, \ldots, d_{2r+1}\}$ are distinct, as (ii) holds. Therefore
	$$ \E[ X_{i_1} \cdots X_{i_{p}} ] \E[ X_{j_1} \cdots X_{j_{q}} ] =0. $$
	Now by the similar calculations as we have done in Case I, second part of (\ref{eqn:lim_cov_SC}) ($p,q$ both odd) will be 
	\begin{align} \label{eqn:p,q_odd}
%	\lim_{n\to\infty} &\Cov\big(w_p,w_q\big)=
    &\lim_{n\to\infty} \frac{1}{n^{\frac{p+q}{2}-1}} \sum_{r=0}^{ \min\{ \frac{p-1}{2},\frac{q-1}{2} \} } \sum_{I_{2r+1}} \big( \E[X_{i_1} \cdots X_{i_{p}} X_{j_1} \cdots X_{j_{q}} ] -\E[ X_{i_1} \cdots X_{i_{p}} ] \E[ X_{j_1} \cdots X_{j_{q}} ]  \big) \nonumber\\
	&  = \lim_{n\to\infty}  \frac{1}{n^{\frac{p+q}{2}-1}} \sum_{r=0}^{ \min\{ \frac{p-1}{2},\frac{q-1}{2} \} } \sum_{I_{2r+1}}  \E[X_{i_1} \cdots X_{i_{p}} X_{j_1} \cdots X_{j_{q}} ]  \nonumber\\
	&  = \sum_{r=0}^{ \min\{ \frac{p-1}{2},\frac{q-1}{2} \} } \frac{b_r}{2^{\frac{p+q-4r-2}{2}}  } \sum_{s=0}^{2r+1}\binom{2r+1}{s}^2 s!(2r+1-s)! \ h_{2r+1}(s) ,
	\end{align}
	 where $h_{2r+1}(s)$ is defined in Result \ref{result:def_h} and $b_r=	\binom{p}{p-2r-1}\binom{p-2r-1}{\frac{p-2r-1}{2}} (\frac{p-2r-1}{2})!\binom{q}{q-2r-1}\binom{q-2r-1}{\frac{q-2r-1}{2}} (\frac{q-2r-1}{2})!.$
%	 Also for $r=0$, second part of (\ref{eqn:lim_cov_SC}) will be zero due to independence of $\{X_i\}$.
	 
%	 \begin{equation}\label{eqn:c_r}
% b_r=	\binom{p}{p-2r-1}\binom{p-2r-1}{\frac{p-2r-1}{2}} (\frac{p-2r-1}{2})!\binom{q}{q-2r-1}\binom{q-2r-1}{\frac{q-2r-1}{2}} (\frac{q-2r-1}{2})!.
%\end{equation}
Now, after combining both the sub-cases  I and II, using (\ref{eqn:p,q_even}) and (\ref{eqn:p,q_odd}) in (\ref{eqn:lim_cov_SC}, we get 
\begin{align} \label{eqn:cov_k=p_odd}
	&\lim_{n\to\infty} \Cov\big(w_p,w_q\big) \nonumber \\ 
	&=  \left\{\begin{array}{ll} 	 
		 	\displaystyle\frac{a_1}{2^{\frac{p+q-4}{2}}  } (\E X^4_1- 1) + \sum_{r=2}^{ \min\{ \frac{p}{2},\frac{q}{2} \} } \frac{a_r}{2^{\frac{p+q-4r}{2}}  } \sum_{s=0}^{2r}\binom{2r}{s}^2 s!(2r-s)! \ h_{2r}(s) & \text{if}\ p, \mbox{} q \mbox{ both are even},\\\\
		 	\displaystyle\sum_{r=0}^{ \min\{ \frac{p-1}{2},\frac{q-1}{2} \} }  \frac{b_r}{2^{\frac{p+q-4r-2}{2}}  } \sum_{s=0}^{2r+1}\binom{2r+1}{s}^2 s!(2r+1-s)! \ h_{2r+1}(s)  & \text{if}\  p, \mbox{} q \mbox{ both are odd},\\\\
			0 & \text{otherwise}. 	 	 
		 	  \end{array}\right.			 	  
	\end{align}
	\noindent \textbf{Case II.} \textbf{$k<p$ and $\ell<q$ :} First recall $\Cov\big(w_p,w_q\big)$ from (\ref{eqn:T_1+T_2_SC}) for $k<p$ and $\ell<q$
		\begin{align*}
		\Cov\big(w_p,w_q\big) \mathbb I_{\{k<p,\ell<q\}}&=\frac{1}{n^{\frac{p+q}{2}-1}} \sum_{k, \ell =0}^{p-1, q-1}\binom{p}{k} \binom{q}{\ell}  \sum_{A_{k}, A_{\ell}} \Big\{\E[ X_0^{p+q-k-\ell}] \E[X_{i_1} \cdots X_{i_{k}} X_{j_1} X_{j_2}\cdots X_{j_{\ell}} ]\nonumber \\
	 &\qquad -\E[X_0^{p-k}]\E[ X_{i_1} \cdots X_{i_{k}} ] \E[X_0^{q-\ell}]\E[ X_{j_1} \cdots X_{j_{\ell}} ]   \Big\}.
\end{align*}			 
	Similar to Case I, we get maximum contribution when $\{i_1,i_2,\ldots,i_{k}\}\cap \{j_1,j_2,\ldots,j_{\ell}\}=\emptyset$ and $(i_1,i_2,\ldots,i_{k}), (j_1,j_2,\ldots,j_{\ell})$ are {\it opposite sign pair matched}. Since from (\ref{eqn:condition}) we have that all moments are bounded, therefore
	\begin{align} \label{eqn:cov_k<p}
	\sum_{A_{k}, A_{\ell}} \E[ X_0^{p+q-k-\ell}] \E[ X_{I_{k}}  X_{J_{\ell}} ] -\E[X_0^{p-k}]\E[X_{I_{k}} ] \E[X_0^{q-\ell}]\E[ X_{J_{\ell}} ] = O(n^{[\frac{k}{2}]+ [\frac{\ell}{2} ]}).
	\end{align}
	 Now using (\ref{eqn:cov_k<p}) and the fact that $\E(X_i)=0$ for each $i=1,2, \ldots,$ we get 
	 \begin{align} \label{eqn:cov_k<p_odd}
	\lim_{n\to\infty} & \Cov\big(w_p,w_q\big) \mathbb I_{\{k<p,\ell<q\}}  \nonumber \\
	   &= \left\{\begin{array}{ll} 	 \nonumber
		 	\displaystyle\lim_{n\to\infty} \frac{pq}{n^{\frac{p+q}{2}-1}}  \sum_{A_{p-1}, A_{q-1}} \E[ X_0^{2}] \E[X_{i_1} \cdots X_{i_{p-1}} X_{j_1} X_{j_2}\cdots X_{j_{q-1}} ] & \text{if}\ p, \mbox{} q \mbox{ both are odd},\\\\
			0 & \text{otherwise}, 	 	 
		 	  \end{array}\right.	\\ 
		 &= \left\{\begin{array}{ll} 
		 	 \displaystyle pq \binom{p-1}{(p-1)/2}\binom{q-1}{(q-1)/2} \l((p-1)/2\r)! \l((q-1)/2\r)!\frac{1}{2^{(\frac{p+q}{2}-1)}} & \text{if}\ p,  q \mbox{ both are odd},\\\\
		 	 0 & \text{otherwise}, 	 	 
		 	  \end{array}\right.	
	 \end{align}
	where the factor $\binom{p-1}{(p-1)/2} \binom{q-1}{(q-1)/2}$ appeared because in $\binom{p-1}{(p-1)/2}$ many ways $\sum_{k=1}^{p-1} \epsilon_k=0$ in one $A_{p-1}$ and $\binom{q-1}{(q-1)/2}$ many ways $\sum_{k=1}^{q-1} \epsilon_k=0$ in $A_{q-1}$.  $\l(\frac{p-1}{2}\r)!$ appeared because for each free choice of  $(p-1)/2$  variables among  $\{i_1,\ldots,i_{p-1}\}$ with positive sign, we can choose rest of the $(p-1)/2$  variables with negative sign in $\l(\frac{p-1}{2}\r)!$ ways to have pair matching. Using  same argument for  $\{j_1,\ldots,j_{q-1}\}$, we get another $\l(\frac{q-1}{2}\r)!$ factor. $\frac{1}{2^{(\frac{p+q}{2}-1)}}$ arises because $1\leq i_k,j_k\leq n/2$.  
	
	Now combining both the Cases, from (\ref{eqn:cov_k=p_odd}), (\ref{eqn:cov_k<p_odd}), we get 
 \begin{align} \label{eqn:cov_odd}
	\lim_{ \substack{{n\to\infty} \\ {n \mbox{ odd} } }} & \Cov\big(w_p,w_q\big) \nonumber \\ 
	&=  \left\{\begin{array}{ll} 	 
		 	\displaystyle\frac{a_1}{2^{\frac{p+q-4}{2}}  } (\E X^4_1- 1) + \sum_{r=2}^{ \min\{ \frac{p}{2},\frac{q}{2} \} } \frac{a_r}{2^{\frac{p+q-4r}{2}}  } \sum_{s=0}^{2r}\binom{2r}{s}^2 s!(2r-s)! \ h_{2r}(s) & \text{if}\ p, \mbox{} q \mbox{ both are even},\\\\
		 	\displaystyle\sum_{r=0}^{ \min\{ \frac{p-1}{2},\frac{q-1}{2} \} }  \frac{b_r}{2^{\frac{p+q-4r-2}{2}}  } \sum_{s=0}^{2r+1}\binom{2r+1}{s}^2 s!(2r+1-s)! \ h_{2r+1}(s) \\		 	
		\displaystyle + pq \binom{p-1}{(p-1)/2}\binom{q-1}{(q-1)/2} \l((p-1)/2\r)! \l((q-1)/2\r)!\frac{1}{2^{(\frac{p+q}{2}-1)}} 	 & \text{if}\ p, \mbox{} q \mbox{ both are odd},\\\\
			0 & \text{otherwise}.	 	 
		 	  \end{array}\right.			 	  
	\end{align}
	
\noindent \textbf{Step 2.} Suppose $n$ is even, then by the trace formula (\ref{trace formula_SC_even}), we get
%	\begin{align*}
%	\E[\Tr(SC_n)^{p}]&
%%	= \E\Big[n\sum_{k=0}^{p}\binom{p}{k}X_0^{p-k}\sum_{A_{k}}\frac{X_{i_1}}{ \sqrt{n}} \cdots \frac{X_{i_{k}}}{ \sqrt{n}} \Big]
%	= \frac{1}{ 2n^ {\frac{p}{2}-1}} \sum_{k=0}^{p}\binom{p}{k}\E \Big[ Y_k  \sum_{A_k} X_{J_{k}} + \tilde{Y}_k \sum_{ \tilde{A}_k}  X_{J_{k}} \Big].	
%	\end{align*}
%	Therefore
	\begin{align} \label{eqn:cov_even}
	& \lim_{n\to\infty}  \Cov\big(w_p,w_q\big) \nonumber \\
	  &= \lim_{n\to\infty} \frac{1}{4n^{\frac{p+q}{2}-1}} \sum_{k, \ell =0}^{p, q}\binom{p}{k} \binom{q}{\ell}   \Big[ \E \{ (Y_k  \sum_{A_k} X_{I_{k}} + \tilde{Y}_k \sum_{ \tilde{A}_k}  X_{I_{k}}) ( Y_\ell  \sum_{A_\ell} X_{J_{\ell}} + \tilde{Y}_\ell \sum_{ \tilde{A}_\ell}  X_{J_{\ell}} ) \}  \nonumber\\
	  &\qquad -   \E[(Y_k  \sum_{A_k} X_{I_{k}} + \tilde{Y}_k \sum_{ \tilde{A}_k}  X_{I_{k}} ]  \E[Y_\ell  \sum_{A_\ell} X_{J_{\ell}} + \tilde{Y}_\ell \sum_{ \tilde{A}_\ell}  X_{J_{\ell}}] \Big] \nonumber \\
&= \lim_{n\to\infty} \frac{1}{4n^{\frac{p+q}{2}-1}} \sum_{k, \ell =0}^{p, q}\binom{p}{k} \binom{q}{\ell}   \Big[ \E  (Y_k Y_\ell \sum_{A_k,A_\ell} X_{I_{k}} X_{J_{\ell}}) + \E (\tilde{Y}_k Y_\ell \sum_{ \tilde{A}_k, A_\ell}  X_{I_{k}}  X_{J_{\ell}}) +\E( Y_k \tilde{Y}_\ell  \sum_{A_k, \tilde{A}_\ell} X_{I_{k}}  X_{J_{\ell}})  \\
&\qquad + \E( \tilde{Y}_k \tilde{Y}_\ell  \sum_{ \tilde{A}_k, \tilde{A}_\ell}  X_{I_{k}} X_{J_{\ell}} ) - \E[(Y_k  \sum_{A_k} X_{I_{k}} + \tilde{Y}_k \sum_{ \tilde{A}_k}  X_{I_{k}} ]  \E[Y_\ell  \sum_{A_\ell} X_{J_{\ell}} + \tilde{Y}_\ell \sum_{ \tilde{A}_\ell}  X_{J_{\ell}}] \Big]. \nonumber 
	\end{align}	 
 By the similar arguments as we have done in Step 1, we can show that right hand side of (\ref{eqn:cov_even}) has non-zero contribution only when $k=p, \ell=q$ with  $\{i_1,i_2,\ldots,i_{k}\}\cap \{j_1,j_2,\ldots,j_{\ell}\} \neq\emptyset$ and $k=p-1, \ell=q-1$ with  $\{i_1,i_2,\ldots,i_{k}\}\cap \{j_1,j_2,\ldots,j_{\ell}\}=\emptyset$. \\\
 
 \noindent \textbf{Case I.} \textbf{$k=p$ and $\ell=q$ :} First recall $Y_k$ and $\tilde{Y}_k$ from (\ref{eqn:Y_k})
 $$Y_k  = {(X_0 +  X_{\frac{n}{2}})}^{p-k} + {(X_0 -  X_{\frac{n}{2}})}^{p-k}, \ \ \tilde{Y}_k = {(X_0 +  X_{\frac{n}{2}})}^{p-k} - {(X_0 -  X_{\frac{n}{2}})}^{p-k}. $$
 Since for $k=p$ and $\ell=q, Y_p=Y_q= 2$ and $\tilde{Y}_p=\tilde{Y}_q=0$. Therefore in this case, (\ref{eqn:cov_even}) will be 
 \begin{align} \label{eqn:cov_even1}
 \lim_{n\to\infty} \Cov\big(w_p,w_q\big) &= \lim_{n\to\infty} \frac{1}{4n^{\frac{p+q}{2}-1}} \Big[ \E  (4\sum_{A_p,A_q} X_{I_{p}} X_{J_{q} }) - \E[ 2\sum_{A_p} X_{I_{p}}] \E[ 2\sum_{A_q} X_{J_{q}}] \Big] \nonumber \\
 &= \lim_{n\to\infty} \frac{1}{n^{\frac{p+q}{2}-1}}  \sum_{A_p,A_q} \E[X_{I_{p}} X_{J_{q}}] - \E[ X_{I_{p}}] \E[X_{J_{q}}].
 \end{align}
 Note that (\ref{eqn:cov_even1}) is same as (\ref{eqn:lim_cov_SC}). Therefore from (\ref{eqn:lim_cov_SC}) and (\ref{eqn:cov_k=p_odd}), (\ref{eqn:cov_even1})  will be
  \begin{align} \label{eqn:cov_k=p_even}
	& \lim_{n\to\infty}  \Cov\big(w_p,w_q\big) \nonumber \\ 
	&=  \left\{\begin{array}{ll} 	 
		 	\frac{a_1}{2^{\frac{p+q-4}{2}}  } (\E X^4_1- 1) + \sum_{r=2}^{ \min\{ \frac{p}{2},\frac{q}{2} \} } \frac{a_r}{2^{\frac{p+q-4r}{2}}  } \sum_{s=0}^{2r}\binom{2r}{s}^2 s!(2r-s)! \ h_{2r}(s) & \text{if}\ p, \mbox{} q \mbox{ both are even},\\\\
		 	\sum_{r=0}^{ \min\{ \frac{p-1}{2},\frac{q-1}{2} \} }  \frac{b_r}{2^{\frac{p+q-4r-2}{2}}  } \sum_{s=0}^{2r+1}\binom{2r+1}{s}^2 s!(2r+1-s)! \ h_{2r+1}(s)  & \text{if}\ p, \mbox{} q \mbox{ both are odd},\\\\
			0 & \text{otherwise}. 	 	 
		 	  \end{array}\right.			 	  
	\end{align}\\\
 
 \noindent \textbf{Case II.} \textbf{$k=p-1$ and $\ell=q-1$ :} 
 Since for $k=p-1$ and $\ell=q-1, Y_{p-1}=Y_{q-1}= 2X_0$ and $\tilde{Y}_{p-1}=\tilde{Y}_{q-1}=2X_{\frac{n}{2}}$. Therefore in this case,  (\ref{eqn:cov_even}) will be 
 \begin{align} \label{eqn:cov,even,k<p}
 & \lim_{n\to\infty}  \Cov\big(w_p,w_q\big)  \\
 &= \lim_{n\to\infty} \frac{1}{4n^{\frac{p+q}{2}-1}} p q\Big[ \E  (4X_0^2 \sum_{A_{p-1},A_{q-1}} X_{I_{p-1}} X_{J_{q-1}}) + \E ( 4 X_0 X_{\frac{n}{2}} \sum_{ \tilde{A}_{p-1}, A_{q-1}}  X_{I_{p-1}} X_{J_{q-1}} ) \nonumber \\
&\qquad  +\E( 4 X_0 X_{\frac{n}{2}}  \sum_{A_{p-1}, \tilde{A}_{q-1}}  X_{I_{p-1}} X_{J_{q-1}} ) + \E( 4 (X_{\frac{n}{2}})^2 \sum_{ \tilde{A}_{p-1}, \tilde{A}_{q-1}}  X_{I_{p-1}} X_{J_{q-1}} ) \nonumber \\
& \qquad  - \E[(2X_0\sum_{A_{p-1}} X_{I_{p-1}} + 2X_{\frac{n}{2}} \sum_{ \tilde{A}_{p-1}}  X_{I_{p-1}} ]  \E[ 2X_0\sum_{A_{q-1}} X_{J_{q-1}} + 2X_{\frac{n}{2}} \sum_{ \tilde{A}_{q-1}}  X_{J_{q-1}}] \Big]. \nonumber 
 \end{align}
 	In this case, we get non-zero contribution when $\{i_1,i_2,\ldots,i_{p-1}\}\cap \{j_1,j_2,\ldots,j_{q-1}\}=\emptyset$ and $(i_1,i_2,\ldots,i_{p-1}), (j_1,j_2,\ldots,j_{q-1})$ are {\it opposite sign pair matched}. Since from (\ref{eqn:condition}) we have that all moments are bounded, therefore
	\begin{align} \label{eqn:cov_k<p,even}
	\E[ X_0^2 \sum_{A_{p-1},A_{q-1}} X_{I_{p-1}} X_{J_{q-1}}] & = O(n^{[\frac{p-1}{2}]+ [\frac{q-1}{2} ]}), \\
   \E [ X_0 X_{\frac{n}{2}} \sum_{ \tilde{A}_{p-1}, A_{q-1}}  X_{I_{p-1}} X_{J_{q-1}} ] & = O(n^{[\frac{p-2}{2}]+ [\frac{q-1}{2} ]}), \nonumber \\
   \E[ X_0 X_{\frac{n}{2}}  \sum_{A_{p-1}, \tilde{A}_{q-1}}  X_{I_{p-1}} X_{J_{q-1}} & = O(n^{[\frac{p-2}{2}]+ [\frac{q-1}{2} ]}),  \nonumber  \\
   \E[ (X_{\frac{n}{2}})^2 \sum_{ \tilde{A}_{p-1}, \tilde{A}_{q-1}}  X_{I_{p-1}} X_{J_{q-1}} ] & = O(n^{[\frac{p-2}{2}]+ [\frac{q-2}{2} ]}).   \nonumber 
	\end{align}
 Now using (\ref{eqn:cov_k<p,even}) and the fact that $\E(X_i)=0$ for each $i=1,2, \ldots$ in (\ref{eqn:cov,even,k<p}), we get 
	 \begin{align} \label{eqn:cov,k<p,even1}
	&\lim_{n\to\infty}  \Cov\big(w_p,w_q\big) \\
	   &= \left\{\begin{array}{ll} 	 \nonumber 
		 	\displaystyle{\lim_{n\to\infty}} \frac{pq}{n^{\frac{p+q}{2}-1}}  \sum_{A_{p-1}, A_{q-1}} \E[ X_0^{2}] \E[X_{i_1} \cdots X_{i_{p-1}} X_{j_1} X_{j_2}\cdots X_{j_{q-1}} ] & \text{if}\ p, \mbox{} q \mbox{ both are odd},\\\\
			0 & \text{otherwise}. 	 	 
		 	  \end{array}\right.	
	 \end{align}
Note that right hand side of (\ref{eqn:cov,k<p,even1}) is same as (\ref{eqn:cov_k<p_odd}). Therefore (\ref{eqn:cov,k<p,even1})  will be

	 \begin{align} \label{eqn:cov_k<p_even}
	& \lim_{n\to\infty}  \Cov\big(w_p,w_q\big)   \\ 
		 &= \left\{\begin{array}{ll}\nonumber 
		 	\displaystyle pq \binom{p-1}{(p-1)/2}\binom{q-1}{(q-1)/2} \l((p-1)/2\r)! \l((q-1)/2\r)!\frac{1}{2^{(\frac{p+q}{2}-1)}} & \text{if}\ p, \mbox{} q \mbox{ both are odd},\\\\
		 	 0 & \text{otherwise}. 	 	 
		 	  \end{array}\right.	
	 \end{align}
Now combining both the Cases, from (\ref{eqn:cov_k=p_even}), (\ref{eqn:cov_k<p_even}), we get
 \begin{align} \label{eqn:cov_even,2}
&\lim_{ \substack{{n\to\infty},\\ {n \mbox{ even} } }}  \Cov\big(w_p,w_q\big) \nonumber \\ 
&=  \left\{\begin{array}{ll} 	 
		 \displaystyle	\frac{a_1}{2^{\frac{p+q-4}{2}}  } (\E X^4_1- 1) + \sum_{r=2}^{ \min\{ \frac{p}{2},\frac{q}{2} \} } \frac{a_r}{2^{\frac{p+q-4r}{2}}  } \sum_{s=0}^{2r}\binom{2r}{s}^2 s!(2r-s)! \ h_{2r}(s) & \text{if}\ p, \mbox{} q \mbox{ both are even},\\\\
		 \displaystyle	\sum_{r=0}^{ \min\{ \frac{p-1}{2},\frac{q-1}{2} \} }  \frac{b_r}{2^{\frac{p+q-4r-2}{2}}  } \sum_{s=0}^{2r+1}\binom{2r+1}{s}^2 s!(2r+1-s)! \ h_{2r+1}(s) \\		 	
	\displaystyle	+ pq \binom{p-1}{(p-1)/2}\binom{q-1}{(q-1)/2} \l((p-1)/2\r)! \l((q-1)/2\r)!\frac{1}{2^{(\frac{p+q}{2}-1)}} 	 & \text{if}\ p, \mbox{} q \mbox{ both are odd},\\\\
			0 & \text{otherwise}. 	 	 
		 	  \end{array}\right.			 	  
	\end{align}	
	Now combining both the Steps, from (\ref{eqn:cov_odd}) and (\ref{eqn:cov_even,2}), we get that for both the odd and even value of $n, \lim_{n\to\infty}  \Cov\big(w_p,w_q\big)$ is same. Therefore 
	  \begin{align} \label{eqn:cov}
	&\lim_{n\to\infty}  \Cov\big(w_p,w_q\big) \nonumber \\ 
	&=  \left\{\begin{array}{ll} 	 
		 	\displaystyle \frac{a_1}{2^{\frac{p+q-4}{2}}  } (\E X^4_1- 1) + \sum_{r=2}^{ \min\{ \frac{p}{2},\frac{q}{2} \} } \frac{a_r}{2^{\frac{p+q-4r}{2}}  } \sum_{s=0}^{2r}\binom{2r}{s}^2 s!(2r-s)! \ h_{2r}(s) & \text{if}\ p, \mbox{} q \mbox{ both are even},\\\\
		\displaystyle 	\sum_{r=0}^{ \min\{ \frac{p-1}{2},\frac{q-1}{2} \} }  \frac{b_r}{2^{\frac{p+q-4r-2}{2}}  } \sum_{s=0}^{2r+1}\binom{2r+1}{s}^2 s!(2r+1-s)! \ h_{2r+1}(s) \\		 	
	\displaystyle	+ pq \binom{p-1}{(p-1)/2}\binom{q-1}{(q-1)/2} \l((p-1)/2\r)! \l((q-1)/2\r)!\frac{1}{2^{(\frac{p+q}{2}-1)}} 	 & \text{if}\ p, \mbox{} q \mbox{ both are odd},\\\\
			0 & \text{otherwise}. 	 	 
		 	  \end{array}\right.			 	  
	\end{align}	
This complete the proof of the theorem \ref{thm:symcircovar}.\end{proof}

\section{Proof of Theorem \ref{thm:symcirpoly}}\label{sec:poly}

First we begin with some notation and definitions. Recall $A_{p}$ and $\tilde{A}_p$ from \eqref{def:A_p_SC} in Section \ref{sec:poly},
\begin{align*}
A_{p} & =\{(j_1,\ldots,j_{p})\suchthat \sum_{i=1}^{p}\epsilon_i j_i=0\; \mbox{(mod n)}, \epsilon_i\in\{+1,-1\}, 1\le j_1,\ldots,j_{p}\le \frac{n}{2}\}, \\
\tilde{A}_{k} &=\{(j_1,\ldots,j_{k})\suchthat \sum_{i=1}^{k}\epsilon_i j_i=0\; \mbox{(mod } \frac{n}{2}) \mbox{ and } \sum_{i=1}^{k}\epsilon_i j_i \neq 0\; \mbox{(mod }n),  \epsilon_i\in\{+1,-1\}, 1\le j_1,\ldots,j_{k}\le \frac{n}{2}\}.
\end{align*} 

For a  vector  $ J = (j_1, j_2, \ldots, j_{p})\in A_{p} \mbox{ or } \tilde{A}_p,$ we define a multi-set $S_J$ as 
\begin{equation}\label{def:S_j}
S_{J} = \{j_1, j_2, \ldots,j_{p}\}.
\end{equation}
\begin{definition}\label{def:connected}
	Two vectors $J =(j_1, j_2, \ldots, j_{p})$ and $J' = (j'_1, j'_2, \ldots, j'_{p})$, where $J \in A_{p}$ and $J' \in A_{q}$, are said to be \textit{connected} if
	$S_{J}\cap S_{J'} \neq \emptyset$.	
\end{definition}
%For $1 \leq i \leq \ell$, suppose $J_i \in A_{2p_i}$. Now, we define cross-matched and self-matched element in $\displaystyle{\cup_{i=1}^{\ell} S_{J_i} }$.
%\begin{definition}\label{def:cross matched}
%	An element in $\displaystyle{\cup_{i=1}^{\ell} S_{J_i} }$ is called \textit{cross-matched} if it appears at least in two distinct $S_{J_i}$. If it appears in $k$ many ${S_{J_i}}^,s$, then we say its \textit{cross-multiplicity} is $k$.
%\end{definition}
%\begin{definition}\label{def:self-matched}
%	An element in $\displaystyle{\cup_{i=1}^{\ell} S_{J_i} }$ is called \textit{self-matched} if it appears more than once in one of $S_{J_i}$. If it appears $k$ many times in $S_{J_i}$, then we say its \textit{self-multiplicity} in $S_{J_i}$ is $k$.
%	
%	An element in $\displaystyle{\cup_{i=1}^{\ell} S_{J_i} }$  can be both self-matched and cross-matched. If an element of $\displaystyle{\cup_{i=1}^{\ell} S_{J_i} }$ has \textit{cross-multiplicity} one, that means, it appears only in one of the ${S_{J_i}}$. 
%\end{definition}	
\begin{definition}\label{def:cluster}
	Given a set of vectors $S= \{J_1, J_2, \ldots, J_\ell \}$, where $J_i \in A_{p_i}$ for $1 \leq i \leq \ell$, a subset $T=\{J_{n_1}, J_{n_2}, \ldots, J_{n_k}\}$ of $S$ is called a \textit{cluster} if it satisfies the following two conditions: 
	\begin{enumerate}
		\item[(i)] For any pair $J_{n_i}, J_{n_j}$ from  $T$ one can find a chain of vectors from 
		$T$, which starts with $J_{n_i}$ and ends with $J_{n_j}$ such that any two neighbouring vectors in the chain are connected.
		\item[(ii)] The subset $\{J_{n_1}, J_{n_2}, \ldots, J_{n_k}\}$ can not  be enlarged to a subset which preserves condition (i).
	\end{enumerate}
\end{definition}
For more details about cluster, we refer the readers to Definition $12$ of \cite{maurya2019process}, where the authors have explained the structure of cluster by using graph.
%We can visualise the notion of cluster using a graph in the following way. We denote the vectors from ${A_{2p_i}}^,s$ by vertices  and the connection between two vectors from ${A_{2p_i}}^,s$ by an edge. Then the clusters are nothing but the connected components in that graph. For example, suppose $\{J_1,J_2,J_3\}$ form a cluster in $\{J_1,J_2,J_3,J_4,J_5\}$. Then in the graph, $\{J_1,J_2,J_3\}$ form a connected component (see Figure \ref{fig:cluster1}). If $\{J_4,J_5\}$  form a different cluster then the graph will have two connected components (see Figure \ref{fig:cluster2}). 
%
%\begin{figure}[h]
%	\centering
%	\includegraphics[height=40mm, width =90mm ]{newcluster1.jpg}
%	\caption{$\{J_1,J_2,J_3\}$ form a cluster in $\{J_1,J_2,J_3,J_4,J_5\}$, and $J_4$ and $J_5$ are not connected to anyone.}\label{fig:cluster1}
%\end{figure}
%
%\begin{figure}[h]
%	\centering
%	\includegraphics[height=40mm, width =90mm ]{newcluster2.jpg}
%	\caption{In both the figures there are two connected components, one consists of $\{J_1,J_2,J_3\}$ and another consists of $\{J_4,J_5\}$.}\label{fig:cluster2}
%\end{figure} 

Now we define a subset $B_{P_\ell}$ of the Cartesian product  $ A_{p_1} \times A_{p_2} \times \cdots \times A_{p_\ell}$ where $A_{p_i}$ is as defined in \eqref{def:A_p_SC}.  
\begin{definition}\label{def:B_{P_l}}
	Let $\ell \geq 2$ and  $P_\ell = (p_1,p_2, \ldots, p_\ell ) $. Now $ B_{P_\ell}$ is a subset of $ A_{p_1} \times A_{p_2} \times \cdots \times A_{p_\ell}$ such that  $ (J_1, J_2, \ldots, J_\ell) \in B_{P_\ell} $ if 
	\begin{enumerate} 
		\item[(i)] $\{J_1, J_2, \ldots, J_\ell\} $ form a cluster, 
		\item[(ii)] each element in  $\displaystyle{\cup_{i=1}^{\ell} S_{J_i} }$ has  multiplicity greater than or equal to two. 
	\end{enumerate}
\end{definition} 
 The next lemma gives us the cardinality of $B_{P_\ell}$.
\begin{lemma}\label{lem:cluster}
	For  $\ell \geq 3 $, 
	\begin{equation}\label{equation:cluster}
	|B_{P_\ell }| = o \big(n^{\frac{p_1+p_2 + \cdots + p_\ell -\ell}{2} }\big).
	\end{equation}
	%where $B_{P_\ell}$ is as defined in Definition $\ref{def:B_{P_l}}$.
\end{lemma}
\begin{proof}
The  proof of this lemma is similar to the proof of Lemma $15$ of \cite{maurya2019process}, where the authors have a different set of constraints on the elements of $A_p$s. But the idea is same. We skip the details here. 
%We can easily prove Lemma \ref{lem:cluster} by using similar ideas.  We leave it to the readers to verify.
\end{proof}
\begin{remark}
	The above lemma is not true if $\ell=2$ and $p_1= p_2$. Suppose $(J_1,J_2)\in B_{P_2}$. Then all $p_1$ entries of $J_1$ may coincides with $p_2(=p_1)$ many entries of $J_2$ and hence 
	$$|B_{P_2}|=O(n^{p_1-1}).$$
	So in this situation, $|B_{P_2}|>o(n^{\frac{p_1+p_2}{2}-1})$.
\end{remark}

The following lemma is an easy consequence of Lemma \ref{lem:cluster}.
\begin{lemma}\label{lem:maincluster}
	Suppose $\{J_1, J_2, \ldots, J_\ell \} $ form a cluster where $J_i\in A_{p_i}$ with $p_i\geq 2$ for $1\leq i\leq \ell$ and $\{X_i\}_{i \geq 1}$ is independent which satisfies (\ref{eqn:condition}). Then for $\ell \geq 3,$
	\begin{equation}\label{equation:maincluster}
	\frac{1}{ n^{\frac{p_1+p_2+ \cdots + p_\ell - \ell}{2}}} \sum_{A_{p_1}, \ldots, A_{p_\ell}} \E\Big[\prod_{k=1}^{\ell}\Big(X_{J_k} - \E(X_{J_k})\Big)\Big] = o(1),
	\end{equation}
	where 
	$$J_k = (j^{k}_{1}, j^{k}_{2}, \ldots, j^{k}_{p_k} ) \ \mbox{and} \ X_{J_k} = X_{j^{k}_{1}} X_{j^{k}_{2}} \cdots X_{j^{k}_{p_k}}.$$	
\end{lemma}
\begin{proof} First observe that   $\E\Big[\prod_{k=1}^{\ell}\Big(X_{J_k} - \E(X_{J_k})\Big)\Big]$ will be non-zero only if each $X_i$ appears at least twice in the collection $\{ X_{j^{k}_{1}}, X_{j^{k}_{2}}, \ldots ,X_{j^{k}_{2p_k}} ; 1\leq k\leq \ell\}$, because $\E(X_i)=0$ for each $i$. Therefore
	\begin{equation}\label{eqn:equality_reduction}
	\sum_{A_{p_1}, \ldots, A_{p_\ell}} \hspace{-3pt}\E\Big[\prod_{k=1}^{\ell}\Big(X_{J_k} - \E(X_{J_k})\Big)\Big]=\sum_{(J_1,\ldots,J_\ell)\in B_{P_\ell}} \hspace{-3pt} \E\Big[\prod_{k=1}^{\ell}\Big(X_{J_k} - \E(X_{J_k})\Big)\Big],
	\end{equation} 
	where $B_{P_\ell}$ as in Definition \ref{def:B_{P_l}}. Since from (\ref{eqn:condition}), we have  
	\begin{equation*}\label{eqn:higher moment finite}
	\E(X^2_i)=1 \mbox{ and } \sup_{i \geq 1}\E(|X_i|^{k})= \alpha_k < \infty \mbox{ for } k \geq 3.
	\end{equation*}  
	 Therefore for $p_1,p_2,\ldots,p_\ell \geq 2$, there exists $\beta_\ell>0$, which depends only on $p_1,p_2,\ldots,p_\ell$, such that  
	\begin{equation}\label{eqn:modulus finite}
	\Big|\E\big[\prod_{k=1}^{\ell}\big(X_{J_k} - \E(X_{J_k})\big)\big]\Big|\leq \beta_\ell
	\end{equation} 
	for all $(J_1, J_2, \ldots, J_\ell)\in A_{p_1}\times A_{p_2}\times \cdots \times A_{p_\ell}$.
	
	Now using \eqref{eqn:equality_reduction} and \eqref{eqn:modulus finite}, we have
	\begin{align*}
	\sum_{A_{p_1}, \ldots, A_{p_\ell}} \Big|\E\big[\prod_{k=1}^{\ell}\big(X_{J_k} - \E(X_{J_k})\big)\big]\Big|
	& \leq  \sum_{(J_1,J_2,\ldots,J_\ell)\in B_{P_\ell}} \beta_{\ell} 
	\ = |B_{p_\ell}| \ \beta_\ell.
	\end{align*}
	By using Lemma \ref{lem:cluster} in above expression, we get
	\begin{align*}
	\sum_{A_{p_1},  \ldots, A_{p_\ell}} \Big|\E\big[\prod_{k=1}^{\ell}\big(X_{J_k} - \E(X_{J_k})\big)\big]\Big| \leq  o \big(n^{\frac{p_1+p_2 + \cdots + p_\ell -\ell}{2} }\big),
	\end{align*}
	and hence 
	\begin{equation*}
	\frac{1}{ n^{\frac{p_1+p_2+ \cdots + p_\ell - \ell}{2}}} \sum_{A_{p_1}, \ldots, A_{p_\ell}} \E\Big[\prod_{k=1}^{\ell}\Big(X_{J_k} - \E(X_{J_k})\Big)\Big] = o(1).
	\end{equation*}
	%where $B_{P_{\ell}}$ as in Definition \ref{def:B_{P_l}}   and $\alpha$ is a constant, which depends on $t_1,t_2,\ldots,t_\ell$ and $p_1,p_2,\ldots,p_\ell$. 
	
	%Now \eqref{equation:maincluster} follows from Lemma \ref{lem:cluster}. 
	This completes the proof of lemma. 
\end{proof}

%\begin{remark} \label{rem:main_cluster}
% Suppose $d_i < p_i$, for some $i=1,2, \ldots, \ell$, then from Lemma \ref{lem:maincluster}, we observe that
%%Suppose $\{J'_1, J'_2, \ldots, J'_\ell \} $ form a cluster where $J'_i\in A_{d_i}$ with $d_i\geq 2$ for $1\leq i\leq \ell$, $\{X_i\}_{i \geq 1}$ is independent and it satisfies (\ref{eqn:condition}). Then for $\ell \geq 3,$
%	\begin{equation*}\label{eqn:maincluster_remark}
%	\frac{1}{ n^{\frac{p_1+p_2+ \cdots + p_\ell - \ell}{2}}} \sum_{A_{d_1}, \ldots, A_{d_\ell}} \E\Big[\prod_{k=1}^{\ell}\Big(X_0^{p_k-d_k}X_{J'_k} - \E(X_0^{p_k-d_k} X_{J'_k})\Big)\Big] = o(1),
%	\end{equation*}
%	where 
%	$$J'_k = (j^{k}_{1}, j^{k}_{2}, \ldots, j^{k}_{d_k} ) \ \mbox{and} \ X_{J'_k} = X_{j^{k}_{1}} X_{j^{k}_{2}} \cdots X_{j^{k}_{d_k}}.$$	
%\end{remark}

\begin{lemma}\label{lem:cluster,decompose}
	Suppose $J_i\in A_{d_i}$ with $d_i\geq 2$ for $1\leq i\leq \ell$ and $\{X_i\}_{i \geq 1}$ is independent which satisfies (\ref{eqn:condition}). Then 
	\begin{align*}
	\sum_{A_{d_1}, \ldots, A_{d_\ell}} \E\Big[\prod_{k=1}^{\ell}\Big(X_{J_k} - \E(X_{J_k})\Big)\Big]
	 &= \left\{\begin{array}{ll} 
		 	O( n^{\frac{d_1+d_2+ \cdots + d_\ell - \ell}{2}}) & \text{if} \ \{ J_1, J_2, \ldots, J_\ell\} \mbox{ decomposes } \\ 
		 	& \mbox{ into clusters of length } 2  \\\\
		 	 o( n^{\frac{d_1+d_2+ \cdots + d_\ell - \ell}{2}}) & \text{otherwise}, 	 
		 	  \end{array}\right.
	\end{align*}
%	\begin{equation}\label{equation:maincluster}
%	 \sum_{A_{d_1}, \ldots, A_{d_\ell}} \E\Big[\prod_{k=1}^{\ell}\Big(X_{J_k} - \E(X_{J_k})\Big)\Big] = O( n^{\frac{d_1+d_2+ \cdots + d_\ell - \ell}{2}}),
%	\end{equation}
	where 
	%$\{ J_1, J_2, \ldots, J_\ell\}$  decomposes into clusters of length 2 
	$$J_k = (j^{k}_{1}, j^{k}_{2}, \ldots, j^{k}_{d_k} ) \ \mbox{and} \ X_{J_k} = X_{j^{k}_{1}} X_{j^{k}_{2}} \cdots X_{j^{k}_{d_k}}.$$	
\end{lemma}
\begin{proof} First observe that  for a fixed $J_1,J_2,\ldots,J_\ell$, if there exists a $k\in\{1,2,\ldots,\ell\}$ such that $J_k$ is not connected with any $J_i$ for $i\neq k$, then 
	$$\E\Big[\prod_{k=1}^{\ell}\Big(X_{J_k} - \E(X_{J_k})\Big)\Big]=0,$$
	due to the independence of $\{X_i\}_{i\geq 1}$.
	
	Therefore for non-zero contribution, $J_1,J_2,\ldots,J_\ell$ must form  clusters with each cluster length greater than or equal to two, that is, each cluster should contain at least two vectors. Suppose $G_1,G_2,\ldots,G_s$ are the clusters formed by  vectors $J_1,J_2,\ldots,J_\ell$ and   $|G_i|\geq 2$ for all $1\leq i \leq s$ where $|G_i|$ denotes the length of the cluster $G_i$. Observe that $\sum_{i=1}^s |G_i|=\ell$.   
	
	If there exists  a cluster $G_j$ among $G_1,G_2,\ldots,G_s$ such that $|G_j|\geq 3$, then from Theorem \ref{thm:symcircovar} and Lemma \ref{lem:maincluster}, we have  
	\begin{align*}
	 \sum_{A_{d_1}, \ldots, A_{d_\ell}} \E\Big[\prod_{k=1}^{\ell}\Big(X_{J_k} - \E(X_{J_k})\Big)\Big]=o(n^{\frac{d_1 + d_2 + \cdots +d_\ell -\ell}{2}}).
	\end{align*}  
	Therefore, if $\ell$ is odd then there will be a cluster of odd length and hence 
	\begin{align*}
	 \sum_{A_{d_1}, \ldots, A_{d_\ell}} \E\Big[\prod_{k=1}^{\ell}\Big(X_{J_k} - \E(X_{J_k})\Big)\Big]=o(n^{\frac{d_1 + d_2 + \cdots +d_\ell -\ell}{2}}).
	\end{align*}  
	
	Similarly, if $\ell$ is even then from Theorem \ref{thm:symcircovar} and Lemma \ref{lem:maincluster}, the contribution is $O(n^{\frac{d_1 + d_2 + \cdots +d_\ell -\ell}{2}})$ only when $\{ J_1, J_2, \ldots, J_\ell\}$  decomposes into clusters of length 2. 
	
This completes the proof of lemma. 
\end{proof}

\begin{remark} \label{rem:main_cluster,even}
%Suppose $X_i$ satisfy the assumptions of Theorem \ref{thm:symcircovar}, then from (\ref{eqn:A,tildeA}), Lemma \ref{lem:maincluster} and Lemma \ref{lem:cluster,decompose}, we get
%\begin{align*}
% \lim_{n\to\infty} &  \frac{1}{ n^{\frac{p_1+p_2+ \cdots + p_\ell - \ell}{2}}}  \sum_{d_1=0}^{p_1} \cdots \sum_{d_\ell=0}^{p_\ell} \binom{p_1}{d_1} \cdots \binom{p_\ell}{d_\ell} \E\Big[ \prod_{k=1}^{\ell} \Big( \sum_{A_{d_k}} Y_{d_k} X_{J_{d_k}} - \E[Y_{d_k} X_{J_{d_k}} ] +  \sum_{ \tilde{A}_{d_k}} \tilde{Y}_{d_k} X_{J_{d_k}}  - \E [\tilde{Y}_{d_k} X_{J_{d_k}}] \Big) \Big]\\
% &=  \lim_{n\to\infty} \frac{1}{n^{\frac{p_1 + p_2 + \cdots +p_\ell -\ell}{2}}} \sum_{A_{p_1}, \ldots, A_{p_\ell}} \E\big[ (X_{J_{p_1}} - \E X_{J_{p_1}})  \cdots (X_{J_{p_\ell}} - \E X_{J_{p_\ell}})\big].
%\end{align*}
 	Suppose $J_i\in F_{d_i}$ with $d_i\geq 2$ for $1\leq i\leq \ell$ and $\{X_i\}_{i \geq 1}$ is independent which satisfies (\ref{eqn:condition}), where $F_{d_i}$ is $A_{d_i}$ or $\tilde{A}_{d_i}$. Then from (\ref{eqn:A,tildeA}) and Lemma \ref{lem:cluster,decompose}, we get
	\begin{align}
	\sum_{F_{d_1}, \ldots, F_{d_\ell}} \E\Big[\prod_{k=1}^{\ell}\Big(X_{J_k} - \E(X_{J_k})\Big)\Big]
	 &= \left\{\begin{array}{ll} 
		 	O( n^{\frac{d_1+d_2+ \cdots + d_\ell - \ell}{2}}) & \text{if} \ \{ J_1, J_2, \ldots, J_\ell\} \mbox{ decomposes into clusters } \\ 
		 	& \mbox{ of length }2 \mbox{ and } F_{d_i} =A_{d_i} \forall \ i=1, 2, \ldots, \ell  \\\\
		 	 o( n^{\frac{d_1+d_2+ \cdots + d_\ell - \ell}{2}}) & \text{otherwise.}	 	 
		 	  \end{array}\right.
	\end{align}
%	where 	
%	$$J_k = (j^{k}_{1}, j^{k}_{2}, \ldots, j^{k}_{d_k} ) \ \mbox{and} \ X_{J_k} = X_{j^{k}_{1}} X_{j^{k}_{2}} \cdots X_{j^{k}_{d_k}}.$$	
\end{remark}
We shall use the above lemmata, Remarks and Theorem \ref{thm:symcircovar} to prove Theorem \ref{thm:symcirpoly}. 

\begin{proof}[Proof of Theorem \ref{thm:symcirpoly}] We use method of moments and  Wick's formula to prove Theorem \ref{thm:symcirpoly}. Recall that from the method of moments, to prove $w_Q \stackrel{d}{\longrightarrow} N(0,\sigma_{Q}^2)$, it is sufficient to show that
	\begin{align*} 
	\lim_{n\to\infty} \E[ (w_Q)^\ell] = \E[ (N(0, \sigma^2_{Q}))^\ell ] \ \ \forall \  \ell=1,2, \ldots.
	\end{align*}
	
	Now to prove above equation, it is enough to show that, for $p_1, p_2, \ldots , p_\ell \geq 2$,
	\begin{align} \label{eqn:moment w_Q}
	\lim_{n\to\infty}\E[w_{p_1}w_{p_2} \cdots w_{p_\ell}]=\E[N_{p_1}N_{p_2} \cdots N_{p_\ell}],
	\end{align}
	where $\{N_{p}\}_{p \geq 1}$ is a centered Gaussian family with covariance $\sigma_{p,q}$, that is, $\E[N_{p},N_{q}]= \sigma_{p,p}$, where $\sigma_{p,q}$ as in (\ref{eqn:sigma_p,q}). Since for odd and even value of $n$, we have different trace formula, therefore we show (\ref{eqn:moment w_Q}) is true for odd and even value of $n$. 
%	We shall show that for both the cases, even and odd value of $n$,
%	
%	in two steps. In Step 1, we calculate limit of $\Cov\big(w_p,w_q\big)$ as $n\to\infty$ with odd $n$ and in Step 2, we calculate limit of $\Cov\big(w_p,w_q\big)$ as $n\to\infty$ with even $n$. We shall show that for both the cases, even and odd value of $n$,  limit of $\Cov\big(w_p,w_q\big)$ is same.
	 
	First suppose $n$ is odd. Since from trace formula (\ref{trace formula_SC_odd}), we have  
	\begin{align*}
	w_{p_k} & = \frac{1}{\sqrt{n}} \Big(\Tr(SC_n)^{p_k} - \E[\Tr(SC_n)^{p_k}]\Big)\\
	%&= \frac{1}{\sqrt{n}}\Big( \frac{n}{n^{\frac{p_k}{2}}} \sum_{A_{p_k}}b_{j_{k,1}}(t_k)\cdots b_{j_{k,p_k}}(t_k) - \E[\frac{n}{n^{\frac{p_k}{2}}} \sum_{A_{p_k}}b_{j_{k,1}}(t_k)\cdots b_{j_{k,p_k}}(t_k)]\Big) \\
%	
%	&= \frac{1}{n^{\frac{p_k -1}{2}}} \sum_{d_k=0}^{p_k}\binom{p_k}{d_k}\sum_{A_{d_k}} \Big( X_0^{p_k-d_k} X_{j_1}\ldots X_{j_{d_k}} - \E[X_0^{p_k-d_k} X_{j_1}\ldots X_{j_{d_k }}]\Big). \\
	&= \frac{1}{n^{\frac{p_k -1}{2}}} \sum_{d_k=0}^{p_k}\binom{p_k}{d_k}\sum_{A_{d_k}} (X_0^{p_k-d_k} X_{J_{d_k}} - \E [X_0^{p_k-d_k} X_{J_{d_k}}] ).
	\end{align*}
%	Note that in the above summation $(j^{k}_{1},j^{k}_{2},\ldots,j^{k}_{p_k})\in A_{p_k}$.
	Therefore 
	\begin{align}\label{eqn:expectation_thm2} 
	& \lim_{n\tends \infty}  \E[w_{p_1}w_{p_2} \cdots w_{p_\ell}] \nonumber\\
	&=\lim_{n\tends \infty}  \frac{1}{n^{\frac{p_1 + p_2 + \cdots +p_\ell -\ell}{2}}} \sum_{d_1=0}^{p_1} \cdots \sum_{d_\ell=0}^{p_\ell} \binom{p_1}{d_1} \cdots \binom{p_\ell}{d_\ell} \sum_{A_{d_1}, \ldots, A_{d_\ell}} \E \Big[ \prod_{k=1}^{\ell}  \Big(X_0^{p_k-d_k} X_{J_{d_k}} - \E [X_0^{p_k-d_k} X_{J_{d_k}}\Big) \Big] \nonumber \\ 
	& = \lim_{n\tends \infty}  \frac{1}{n^{\frac{p_1 + p_2 + \cdots +p_\ell -\ell}{2}}} \sum_{A_{p_1}, \ldots, A_{p_\ell}} \E\Big[\prod_{k=1}^{\ell}\Big(X_{J_k} - \E(X_{J_k})\Big)\Big],
	\end{align}
	where the last equality comes due to Lemma \ref{lem:cluster,decompose}. Because %when $d_i <p_i$ for some $i=1, 2, \ldots,\ell$, then
	$$\sum_{A_{d_1}, \ldots, A_{d_\ell}} \E \Big[ \prod_{k=1}^{\ell}  \Big(X_0^{p_k-d_k} X_{J_{d_k}} - \E [X_0^{p_k-d_k} X_{J_{d_k}}\Big) \Big]  \leq O(n^{\frac{d_1 + d_2 + \cdots +d_\ell -\ell}{2}}).$$
	
%	Now for a fixed $J_1,J_2,\ldots,J_\ell$, if there exists a $k\in\{1,2,\ldots,\ell\}$ such that $J_k$ is not connected with any $J_i$ for $i\neq k$, then 
%	$$\E\big[ (X_{J_1} - \E X_{J_1}) (X_{J_2} - \E X_{J_2})  \cdots (X_{J_\ell} - \E X_{J_\ell})\big]=0$$
%	due to the independence of $\{X_i\}_{i\geq 1}$.
%	
%	Therefore $J_1,J_2,\ldots,J_\ell$ must form  clusters with each cluster length greater than or equal to two, that is, each cluster should contain at least two vectors. Suppose $G_1,G_2,\ldots,G_s$ are the clusters formed by  vectors $J_1,J_2,\ldots,J_\ell$ and   $|G_i|\geq 2$ for all $1\leq i \leq s$ where $|G_i|$ denotes the length of the cluster $G_i$. Observe that $\sum_{i=1}^s |G_i|=\ell$.   
%	
%	If there exists  a cluster $G_j$ among $G_1,G_2,\ldots,G_s$ such that $|G_j|\geq 3$, then from Theorem \ref{thm:symcircovar} and Lemma \ref{lem:maincluster}, we have  
%	\begin{align*}
%	\frac{1}{n^{\frac{p_1 + p_2 + \cdots +p_\ell -\ell}{2}}} \sum_{A_{p_1}, \ldots, A_{p_\ell}} \E\big[ (X_{J_1} - \E X_{J_1}) \cdots (X_{J_\ell} - \E X_{J_\ell})\big]=o(1).
%	\end{align*}  
%	Therefore, if $\ell$ is odd then there will be a cluster of odd length and hence 
%	$$ \lim_{n\tends \infty}  \E[w_{p_1} w_{p_2}\cdots w_{p_\ell}] = 0.$$ 
%	
%	Similarly, if $\ell$ is even then the contribution due to $\{ J_1, J_2, \ldots, J_\ell \}$ to  $ \E[w_{p_1} w_{p_2}\cdots w_{p_\ell}]$ is $O(1)$ only when $\{ J_1, J_2, \ldots, J_\ell\}$  decomposes into clusters of length 2.
 Now combining Lemma \ref{lem:cluster,decompose} for $d_i=p_i$ and \eqref{eqn:expectation_thm2}, we get
	\begin{align*} \label{eq:multisplit}
	& \quad \lim_{n\tends \infty}  \E[w_{p_1} w_{p_2}\cdots w_{p_\ell}]\\
	& \quad  =\lim_{n\to\infty} \frac{1}{n^{\frac{p_1 + p_2 + \cdots +p_\ell -\ell}{2}}} \sum_{A_{p_1}, \ldots, A_{p_\ell}} \E\Big[\prod_{k=1}^{\ell}\Big(X_{J_k} - \E(X_{J_k})\Big)\Big]\\
	& \quad  =\lim_{n\to\infty} \frac{1}{n^{\frac{p_1 + p_2 + \cdots +p_\ell -\ell}{2}}} \sum_{\pi \in \mathcal P_2(\ell)} \prod_{i=1}^{\frac{\ell}{2}}  \sum_{A_{p_{y(i)}},\ A_{p_{z(i)}}} \E\big[ (X_{J_{y(i)}} - \E X_{J_{y(i)}}) (X_{J_{z(i)}} - \E X_{J_{z(i)}})\big],
	\end{align*}
	where  $\pi = \big\{ \{y(1), z(1) \}, \ldots , \{y(\frac{\ell}{2}), z(\frac{\ell}{2})  \} \big\}\in \mathcal P_2(\ell)$ and $\mathcal P_2(\ell)$ is the set of all pair partition of $ \{1, 2, \ldots, \ell\} $. Using Theorem \ref{thm:symcircovar}, from the last equation, we get
	\begin{equation}\label{eqn:product of expectation}
	\lim_{n\tends \infty}  \E[w_{p_1}w_{p_2} \cdots w_{p_\ell}]
	=\sum_{\pi \in P_2(\ell)} \prod_{i=1}^{\frac{\ell}{2}} \lim_{n\tends \infty} \E[w_{p_{y(i)}} w_{p_{z(i)}}].
	\end{equation}
	%&= \lim_{n\tends \infty} \sum_{\pi \in P_2(\ell)} \prod_{i=1}^{\frac{\ell}{2}} \E[w_{p_{y(i)}} (t_{y(i)}) w_{p_{z(i)}} (t_{z(i)})] \nonumber \\
    Since from Theorem \ref{thm:symcircovar}, we have 
	$$\lim_{n\to\infty}\E(w_p w_q) = \sigma_{p,q} (= \E(N_p N_q)).$$
	Therefore using Wick's formula,  from \eqref{eqn:product of expectation} we get 
	\begin{align*}
	\lim_{ \substack{{n\to\infty}  \\ {n \mbox{ odd} } }} \E[w_{p_1}w_{p_2} \cdots w_{p_\ell}]
	&=\sum_{\pi \in P_2(\ell)} \prod_{i=1}^{\frac{\ell}{2}} \lim_{n\tends \infty} \E[w_{p_{y(i)}} w_{p_{z(i)}}] \\
	& =\sum_{\pi \in \mathcal P_2(\ell)} \prod_{i=1}^{\frac{\ell}{2}} \E[N_{p_{y(i)}} N_{p_{z(i)}} ] \\
	&=\E[ N_{p_1}N_{p_2} \cdots N_{p_\ell} ].
	\end{align*} 
	
	Now suppose $n$ is even. Then by using trace formula (\ref{trace formula_SC_even}), we get
  \begin{align*}
	w_{p_k} & = \frac{1}{n^{\frac{p_k -1}{2}}} \sum_{d_k=0}^{p_k}\binom{p_k}{d_k}  \Big[ Y_k  \sum_{A_{d_k}} X_{J_{d_k}} + \tilde{Y}_{d_k} \sum_{ \tilde{A}_{d_k}}  X_{J_{d_k}}  - \E [Y_{d_k}  \sum_{A_{d_k}} X_{J_{d_k}} + \tilde{Y}_{d_k} \sum_{ \tilde{A}_{d_k}}  X_{J_{d_k}} ] \Big] \\
	 & = \frac{1}{n^{\frac{p_k -1}{2}}} \sum_{d_k=0}^{p_k}\binom{p_k}{d_k}  \Big[ \sum_{A_{d_k}} Y_{d_k} X_{J_{d_k}} - \E[Y_{d_k} X_{J_{d_k}} ] +  \sum_{ \tilde{A}_{d_k}} \tilde{Y}_{d_k} X_{J_{d_k}}  - \E [\tilde{Y}_{d_k} X_{J_{d_k}}] \Big],\\
	\end{align*}
  and therefore 
	\begin{align}\label{eqn:expectation_thm2,even} 
	& \lim_{n\tends \infty}   \E[w_{p_1}w_{p_2} \cdots w_{p_\ell}] \nonumber\\
	&= \lim_{n\tends \infty}  \frac{1}{n^{\frac{p_1 + p_2 + \cdots +p_\ell -\ell}{2}}} \sum_{d_1=0}^{p_1} \cdots \sum_{d_\ell=0}^{p_\ell} \binom{p_1}{d_1} \cdots \binom{p_\ell}{d_\ell} \E\Big[ \prod_{k=1}^{\ell} \Big( \sum_{A_{d_k}} Y_{d_k} X_{J_{d_k}} - \E[Y_{d_k} X_{J_{d_k}} ] \nonumber \\
	& \qquad +  \sum_{ \tilde{A}_{d_k}} \tilde{Y}_{d_k} X_{J_{d_k}}  - \E [\tilde{Y}_{d_k} X_{J_{d_k}}] \Big) \Big]\nonumber \\
	& = \lim_{n\tends \infty}  \frac{1}{n^{\frac{p_1 + p_2 + \cdots +p_\ell -\ell}{2}}} \sum_{A_{p_1}, \ldots, A_{p_\ell}} \E\Big[\prod_{k=1}^{\ell}\Big(X_{J_k} - \E(X_{J_k})\Big)\Big],
	\end{align}
	where the last equality comes due to Lemma \ref{lem:cluster,decompose} and Remark \ref{rem:main_cluster,even}. Since (\ref{eqn:expectation_thm2,even}) is same as (\ref{eqn:expectation_thm2}), therefore  by the the similar calculation as we have done for $n$ odd case, we get
%	\begin{align*}
%	\lim_{ \substack{{n\to\infty}  \\ {n \mbox{ odd} } }} \E[w_{p_1}w_{p_2} \cdots w_{p_\ell}]
%	&=\E[ N_{p_1}N_{p_2} \cdots N_{p_\ell} ].
%	\end{align*} 
	$$\lim_{ \substack{{n\to\infty}  \\ {n \mbox{ even} } }} \E[w_{p_1}w_{p_2} \cdots w_{p_\ell}]
	=\E[ N_{p_1}N_{p_2} \cdots N_{p_\ell} ].$$
	%Now odd and even case of $n$ 
	This completes the proof of Theorem \ref{thm:symcirpoly} after combining  odd and even cases of $n$.
\end{proof}

%We thank both the referees for their useful suggestions.
%\bibliography{shambhubib}
%\bibliographystyle{amsplain}

\providecommand{\bysame}{\leavevmode\hbox to3em{\hrulefill}\thinspace}
\providecommand{\MR}{\relax\ifhmode\unskip\space\fi MR }
\providecommand{\MRhref}[2]{%
  \href{http://www.ams.org/mathscinet-getitem?mr=#1}{#2}
}
\providecommand{\href}[2]{#2}

\end{document}